\documentclass[a4paper,10pt]{article}
\usepackage{amssymb,latexsym,amsmath,enumerate,verbatim,amsfonts,amsthm}

\newtheorem{lemma}{Lemma}[section]
\newtheorem{definition}{Definition}[section]

\newtheorem{theorem}{Theorem}[section]
\newtheorem{corollary}{Corollary}[section]
\newtheorem{proposition}{Proposition}[section]
\newtheorem{remark}{Remark}[section]
\newtheorem{example}{Example}[section]

\begin{document}
\title{\Large{\sf{\sf{A Tensor Analogy of Yuan's Theorem of the Alternative and Polynomial Optimization with Sign structure \thanks{Corresponding author: Guoyin Li, Department of Applied Mathematics, University of New South Wales,
Sydney 2052, Australia. (Email: g.li@unsw.edu.au)}}}}}
 \author{Shenglong Hu \footnote{Department of Mathematics, School of Science, Tianjin University,
Tianjin, China.  E-mail:  timhu@tju.edu.cn; tim.hu@connect.polyu.hk (S. Hu)} \quad Guoyin Li \footnote{Department of Applied Mathematics, University of New South Wales,
Sydney 2052, Australia. E-mail: g.li@unsw.edu.au (G. Li)} \quad Liqun Qi\footnote{Department of Applied
 Mathematics, The Hong Kong Polytechnic University, Hung Hom, Kowloon, Hong
 Kong.
 E-mail: maqilq@polyu.edu.hk (L. Qi)
}}

\date{}
\maketitle

\begin{abstract}
\noindent Yuan's theorem of the alternative is an important theoretical tool in optimization, which provides a checkable certificate for the infeasibility of a strict inequality system involving two homogeneous quadratic functions. In this paper, we provide a tractable extension of Yuan's theorem of the alternative to the symmetric tensor setting. As an application, we establish that the optimal value of a class of nonconvex polynomial optimization problems with suitable sign structure (or more explicitly, with essentially non-positive coefficients) can be computed by a related convex conic programming problem, and the optimal solution of these nonconvex polynomial optimization problems can be recovered from the corresponding solution of the convex conic programming problem. Moreover, we obtain that this class of nonconvex polynomial optimization problems enjoy exact sum-of-squares relaxation, and so, can be solved via a single semidefinite programming problem.

\bigskip


\bigskip

\end{abstract}

\section{Introduction}
Alternative theorems for arbitrary finite systems of linear or convex inequalities
have played key roles in the development of optimality conditions for continuous optimization
problems. Although these theorems are generally not valid for an
arbitrary finite system of (possibly nonconvex) quadratic inequalities, recent research has established alternative
theorems for quadratic systems involving two inequalities or arbitrary inequalities involving suitable sign structure. For
instance, a theorem of the alternative of Gordan type for a strict inequality system of two
homogeneous quadratic functions has been given in [1]. This theorem is often referred as Yuan's theorem of the alternative.
This theorem provides a checkable certificate for the infeasibility of a strict inequality system involving two homogeneous quadratic functions, and plays an important role
in the convergence analysis of the trust region method. Recently, it was also established in [2] that this theorem of the alternative is equivalent to another popular
result called S-lemma, which is an important tool in quadratic optimization, optimal control and robust optimization [3-6]. 

Because of the importance of this theorem of the alternative, researchers
have attempted to extend it to systems with more than two quadratic
functions. In particular, [7,8] showed that, under a
positive definite regularity condition, Yuan's theorem of the alternative
continues to hold for an inequality system with three homogeneous quadratic
functions. They also provided examples illustrating that, in general,
the regularity condition cannot be dropped. Moreover,
[9,10] (see also [11]) established an extension of
Yuan's theorem of the alternative to inequality systems involving finitely
many quadratic functions, under the condition that all the Hessian matrices of the quadratic functions
have non-positive off-diagonal elements (up to a nonsingular transformation).
%
This result can be regarded as an extension of Yuan's theorem of the alternative as its assumption becomes superfluous in the case when only two quadratic forms are involved (see [10, Remark 2.3]).

{The purpose of this paper is to extend Yuan's alternative
theorem to symmetric tensors and to provide an application to
nonconvex polynomial optimization}. Tensor (or hypermatrix) is a
multilinear generalization of the concept of matrix.
Recently, Lim [12] and Qi [13]  independently introduced the concept of eigenvalues and singular values
for tensors of higher order.
After this, 
a reasonably
complete and consistent theory of eigenvalues and singular values
for tensors of higher order has been developed in the past few years, which generalizes the theory of matrix eigenvalues and singular values in various manners and extent.
Moreover, numerical study on tensors also has attracted a lot of researchers due to its wide applications in polynomial optimization
[14-19], space tensor programming [20,21],
   spectral hypergraph theory [22-25], high-order Markov chain [26], signal processing [27,28] and image science [29]. In particular, various
efficient numerical schemes have been proposed to find the low rank
approximations of a tensor and  the eigenvalues/eigenvectors of a
tensor with specific structure (cf. [30-37]).

The contribution of this paper are as follows:
\begin{itemize}
 \item Firstly, we provide a tractable extension of Yuan's theorem of the alternative  (Theorem \ref{th:1}) and homogeneous S-lemma  (Corollary \ref{cor:S-lemma}) to the symmetric
tensor setting. We achieve this by exploiting two important features of a special class of tensors (called essentially non-positive tensors): hidden convexity and numerical checkability.
\item Secondly, we establish that the optimal value of a class of nonconvex polynomial optimization problems with suitable sign structure (or more explicitly, essentially nonpositive coefficients) can be computed by a related convex conic programming problem, and the optimal solution of these nonconvex polynomial optimization problems can be recovered from the corresponding solution of the convex conic programming problem. Moreover, we obtain that this class of  nonconvex polynomial optimization problems enjoy exact sum-of-squares relaxation, and so, can be solved via a semidefinite programming problem.
\end{itemize}

The organization of this paper is as follows. In Section 2, we
recall some basic facts of tensors and polynomials, and establish
some basic geometric properties of positive semidefinite tensor
cones. In Section 3, we provide a tractable extension of Yuan's
theorem of the alternative and homogeneous S-lemma to the symmetric tensor
setting. In Section 4, we apply the new theorem of the alternative to
obtain exact conic programming relaxation for nonconvex polynomial
optimization problems with essentially nonpositive coefficients. We
also obtain that these class of  nonconvex polynomial optimization
problems enjoy exact sum-of-squares relaxation. Finally, we
conclude this paper and present some possible future
research topics.
\section{Preliminaries: Positive Semidefinite Tensors}
\subsection{Notations and Basic Facts}
We first fix some notations and recall some basic facts of tensors and polynomials.
We denote the $n$-dimensional Euclidean space as $\mathbb{R}^n$. For $x_1,x_2 \in \mathbb{R}^n$ (as column vectors),  $\langle x_1,x_2\rangle$ denotes
 the inner product between $x_1$ and $x_2$ and is given by $\langle x_1, x_2 \rangle:=x_1^Tx_2$.  Moreover, for all $x \in \mathbb{R}^n$, the norm of $x$ is denoted by $\|x\|$ and is given by $\|x\|:=(\langle x, x \rangle)^{1/2}$.

Let $n \in \mathbb{N}$ and let $m$ be an even number. An $m$th-order $n$-dimensional tensor $\mathcal {A}$ consists of
$n^m$ entries in real number:
$\mathcal {A}=(\mathcal{A}_{i_1i_2\cdots i_m}), \ \ \mathcal{A}_{i_1i_2\cdots i_m} \in
\mathbb{R}, \ \  1\leq i_1,i_2,\cdots, i_m \leq n$.
We say a tensor $\mathcal{A}$ is symmetric if the value of
$\mathcal{A}_{i_1i_2 \cdots i_m}$ is invariant under any permutation
of its indices $\{i_1,i_2,\cdots,i_m\}$. When $m=2$, a
symmetric tensor is nothing but a symmetric matrix.  Consider
$$S_{m,n}:=\{\mathcal{A}: \mathcal{A} \mbox{ is an } m\mbox{th-order }  n\mbox{-dimensional} \mbox{  symmetric tensor}\}. $$
Clearly, $S_{m,n}$ is a vector space under the addition and multiplication
defined as below: for any $t \in \mathbb{R}$,
$\mathcal{A}=(\mathcal{A}_{i_1 \ldots i_m})_{1 \le i_1,\ldots,i_m
\le n}$ and $\mathcal{B}=(\mathcal{B}_{i_1 \ldots i_m})_{1 \le
i_1,\ldots,i_m \le n}$
\[
\mathcal{A}+\mathcal{B}=(\mathcal{A}_{i_1 \ldots i_m}+\mathcal{B}_{i_1 \ldots i_m})_{1
\le i_1,\ldots,i_m \le n} \mbox{ and } t
\mathcal{A}=(t\mathcal{A}_{i_1 \ldots i_m})_{1 \le i_1,\ldots,i_m
\le n}.
\]
For each $\mathcal{A},\mathcal{B} \in S_{m,n}$, we define the inner
product by
\[
\langle
\mathcal{A},\mathcal{B}\rangle:=\sum_{i_1,\ldots,i_m=1}^{n}\mathcal{A}_{i_1 \ldots
i_m}\mathcal{B}_{i_1 \ldots i_m}.
\]
The corresponding norm is defined by $\displaystyle
\|\mathcal{A}\|=(\langle
\mathcal{A},\mathcal{A}\rangle)^{1/2}=\big(\sum_{i_1,\ldots,i_m=1}^{n}(\mathcal{A}_{i_1 \ldots i_m})^2\big)^{1/2}$.
For a vector $x\in
\mathbb{R}^n$, we use $x_i$ to denote its $i$th component.
Moreover, for a  vector $x\in
\mathbb{R}^n$, we use $x^{\otimes m}$ to denote the $m$th-order $n$-dimensional symmetric rank one tensor induced by $x$, i.e.,
\[
(x^{\otimes m})_{i_1 i_2\ldots i_m}=x_{i_1}x_{i_2}\ldots x_{i_m}, \ \forall \, i_1,\ldots,i_m \in \{1,\ldots,n\}.
\]

We now collect some basic facts on real polynomials. Recall that $f\colon\mathbb{R}^n\rightarrow\mathbb{R}$ is a {(real) polynomial} if there exists a number $d\in\mathbb{N}$ such that
$$
f(x):=\sum_{0\le|\alpha|\le d}f_{\alpha}x^{\alpha},
$$
where $f_{\alpha}\in \mathbb{R}$, $x=(x_1,\cdots,x_n)$, $x^{\alpha}:=x_1^{\alpha_1}\cdots x_n^{\alpha_n}$, $\alpha_i\in\mathbb{N}\cup \{0\}$, and $|\alpha|:=\sum_{j=1}^{n}\alpha_j$.
The corresponding number $d$ is called the {  degree} of $f$, and is denoted by ${\rm deg}f$. For a degree $d$ real polynomial $f$ on $\mathbb{R}^n$ with the form
$\displaystyle f(x)=\sum_{0\le|\alpha|\le d}f_{\alpha}x^{\alpha}$,
its {  canonical homogenization} $\tilde{f}$ is a homogeneous polynomial on $\mathbb{R}^{n+1}$ with degree $d$ given by
\[
\tilde{f}(x,t)=\sum_{0\le|\alpha|\le d}f_{\alpha}x^{\alpha}t^{d-|\alpha|}.
\]
A real polynomial $f$ is called a sum-of-squares (SOS) polynomial  if there exist $r \in \mathbb{N}$ and real polynomials $f_j$, $j=1,\ldots,r$, such that $f=\sum_{j=1}^rf_j^2$.
An important property of the sum of squares of polynomials is that checking a polynomial is sum of squares or not, is equivalent to solving a semi-definite linear programming problem (cf. [38-40]). 

Finally, we note that an $m$th-order $n$-dimensional
symmetric tensor uniquely defines an $m$th degree homogeneous
real polynomial $f_{\mathcal{A}}$ on $\mathbb{R}^n$: for all $x=(x_1,\ldots,x_n)^T
\in \mathbb{R}^n$, \[
f_{\mathcal{A}}(x)=\langle \mathcal{A}, x^{\otimes m}\rangle:= \sum_{i_1,\ldots,i_m=1}^{n} \mathcal{A}_{i_1i_2\cdots
i_m}x_{i_1}x_{i_2}\ldots x_{i_m}.\]
Conversely, any $m$th degree homogeneous
polynomial function $f$ on $\mathbb{R}^n$ also uniquely corresponds a symmetric tensor.
Let $n \in \mathbb{N}$ and $m$ be an even number.  Define $I(m,n)=\left(\begin{array}{c}
n+m-1 \\
n-1
\end{array} \right)$. It is known that the space consists of all homogeneous polynomials on $\mathbb{R}^n$ with degree $m$ is a finite dimensional
space with dimension $I(m,n)$. Note that each $\mathcal{A} \in S_{m,n}$ uniquely corresponds a homogeneous polynomial on $\mathbb{R}^n$ with degree $m$. It follows that ${\rm dim}S_{m,n}=I(m,n)$.

\subsection{Positive Semidefinite Tensors and
Their Associated Cones}


\begin{definition}{\bf (PSD tensor cone and SOS tensor cone)}
 Let $m$ be an even number and $n \in \mathbb{N}$. We say an $m$th-order $n$-dimensional symmetric tensor $\mathcal{A}$ is
 \begin{itemize}
 \item[{\rm (i)}]  {\it a positive semi-definite (PSD) tensor} iff
 $f_{\mathcal{A}}(x):=\langle \mathcal{A}, x^{\otimes m} \rangle \ge 0$ for all $x \in \mathbb{R}^n$;
 \item[{\rm (ii)}] {\it a sum-of-squares (SOS) tensor} iff
 $f_{\mathcal{A}}(x):=\langle \mathcal{A}, x^{\otimes m} \rangle$ is a sum-of-squares polynomial.
 \end{itemize}
 Moreover, we define the PSD tensor cone ${\rm PSD}_{m,n}$ (resp. SOS tensor cone ${\rm SOS}_{m,n}$) to be the set consisting of all positive semi-definite (resp. sum-of-squares) $m$th-order $n$-dimensional symmetric tensors.
 \end{definition}


Note that any sum-of-squares polynomial must take non-negative
values. So, ${\rm SOS}_{m,n} \subseteq {\rm PSD}_{m,n}$ for each $m
\in \mathbb{N}$ and $n \in \mathbb{N}$. It is known that [41,42] 
that ${\rm SOS}_{m,n} = {\rm PSD}_{m,n}$ in one of the following three cases: $n=1$; $m = 2$;  $n = 3 \mbox{ and } m = 4$. Moreover, if $m=2$, then
${\rm PSD}_{m,n}$ and ${\rm SOS}_{m,n}$ are equal, and both collapse
to the positive semi-definite matrix cone.    On the other hand, the
inclusion ${\rm SOS}_{m,n} \subseteq {\rm PSD}_{m,n}$ is strict in general. Indeed, let $f_M$ be the homogeneous
Motzkin polynomial
\[f_M(x)=x_3^6 + x_1^2x_2^4 + x_1^4x_2^2-3x_1^2x_2^2x_3^2.
\]
It is
known that {\rm (cf. [43]
)}, $f_M$ takes non-negative value (by the Arithmetic-Geometric inequality), and it is not a sum-of-squares polynomial.
Let $\mathcal{A}_M$ be the symmetric tensor
associated to $f_M$ in the sense that $f_M(x)=\langle \mathcal{A}_M, x^{\otimes 6}\rangle$. Then, we see that $\mathcal{A}_M \in {\rm PSD}_{6,3} \backslash {\rm SOS}_{6,3}$.

Below, we identify a class of tensors with suitable sign structure such that they are sum-of-squares whenever they are positive semidefinite.
\begin{definition}{\bf (Essentially nonpositive/non-negative tensor)}
Define the index set $I$ by $$I:=\{(i,i,\ldots,i) \in \mathbb{N}^m: 1 \le i \le n\}.$$ We say an $m$th-order $n$-dimensional tensor $\mathcal{A}$ is
 \begin{itemize}
 \item[{\rm (i)}]   essentially non-negative iff $\mathcal{A}_{i_1,\ldots,i_m} \ge 0$ for all $\{i_1,\ldots,i_m\} \notin I$.
 \item[{\rm (ii)}]   essentially nonpositive iff $\mathcal{A}_{i_1,\ldots,i_m} \le 0$ for all $\{i_1,\ldots,i_m\} \notin I$.
 \end{itemize} Define $E_{m,n}$ as the set consisting of all essentially nonpositive tensor, that is, $$E_{m,n}:=\{\mathcal{A} \in S_{m,n}: \mathcal{A} \mbox{ is essentially nonpositive}\}.$$
 \end{definition}

In the special case when the order $m=2$,
the definition of essentially nonpositive tensor reduces to the notion of a $Z$-matrix. The class of essentially non-negative tensors was introduced in
[36] (see also [31]), and some interesting
log-convexity results were discussed there.   One interesting example
of essentially nonpositive tensors is the Laplacian tensor of a
hypergraph, which was examined in detail recently in
[23-25]. From the definition,  any
tensor with non-negative entries is essentially non-negative, while the
converse may not be true in general.

For any essentially nonpositive tensor $\mathcal{A}$, we establish that it is positive semi-definite if and only if it is sum-of-squares. To do this, we first recall some definitions and a useful lemma.

Consider a homogeneous polynomial $f(x)=\sum_{\alpha}f_{\alpha}x^{\alpha}$ with degree $m$ ($m$ is an even number). 
Let $f_{m,i}$ be the coefficient associated with $x_i^{m}$ and  \begin{equation} \Omega_f:=\{\alpha:=(\alpha_1,\ldots,\alpha_n) \in (\mathbb{N} \cup \{0\})^n: f_{\alpha} \neq 0 \mbox{ and } \alpha \neq m e_i, \ i=1,\ldots,n\},\end{equation}
where $e_i$ be the vector whose $i$th component is one and all the other components are zero. We note that
\[
f(x)=\sum_{i=1}^n f_{m,i} x_i^{m}+\sum_{\alpha \in \Omega_f}f_{\alpha}x^{\alpha}.
\] Recall that $2\mathbb{N}$ denotes the set consisting of all the even numbers. Define \begin{equation}\label{eq:Deltaf}
\Delta_f:=\{\alpha=(\alpha_1,\ldots,\alpha_n) \in \Omega_f: f_{\alpha} < 0 \mbox{ or } \alpha \notin (2\mathbb{N} \cup \{0\})^n\}.
\end{equation}
We associate to $f$ a new homogeneous polynomial $\hat{f}$, given by
\[
\hat{f}(x)=\sum_{i=1}^n f_{m,i} \, x_i^{m}-\sum_{\alpha \in \Delta_f}|f_{\alpha}|x^{\alpha}.
\]
We now recall the following useful lemma, which provides a test for verifying whether $f$ is a sum of squares polynomial or not in terms of the nonnegativity of the new homogeneous function $\hat{f}$.
\begin{lemma} {\rm ([44, Corollary 2.8])} \label{lemma:2.1}
Let $f$ be a homogeneous polynomial of degree
$m$ where $m$ is an even number. If $\hat{f}$ is a polynomial which always takes non-negative values, then $f$ is a sum-of-squares polynomial.
\end{lemma}

We are now ready to state the fact that, any essentially nonpositive tensor is positive semi-definite if and only if it is sum-of-squares. This fact was essentially established in [31]. For the self-containment purpose, its proof is provided in the appendix for the reader's convenience.
\begin{proposition}\label{prop:0.1}
It holds that ${\rm PSD}_{m,n} \cap E_{m,n}={\rm SOS}_{m,n} \cap E_{m,n}$.
\end{proposition}

 Next, we study the dual cone of ${\rm PSD}_{m,n}$. Recall that for a given closed and convex cone $C$ in $S_{m,n}$, its  dual cone (or positive polar) $C^{\oplus}$ is defined as
\[
C^{\oplus}:=\{\mathcal{X} \in S_{m,n}: \langle \mathcal{X},\mathcal{C} \rangle \ge 0 \mbox{ for all } \mathcal{C} \in C\}.
\] To establish the dual cone of ${\rm PSD}_{m,n}$, we first define a set which is the convex hull of all rank one tensors.
\begin{definition}
 Let $m$ be an even number and $n \in \mathbb{N}$. we define the
set $U_{m,n}$ as the convex hull of all  $m$th-order $n$-dimensional symmetric rank one tensors, that is,
\[
U_{m,n}:={\rm conv}\{x^{\otimes m}: x \in \mathbb{R}^n\}.
\]
\end{definition}

Next, we justify that the set $U_{m,n}$ is indeed a closed convex cone.
\begin{lemma}
 Let $m$ be an even number and $n \in \mathbb{N}$. Then,  $U_{m,n}$ is a closed and convex cone with dimension at most $I(m,n)$.
\end{lemma}
\begin{proof}
From the definition, $U_{m,n}$ is a convex cone. Note that $U_{m,n}
\subseteq S_{m,n}$ and $S_{m,n}$ is of dimension $I_{m,n}$. So,
$U_{m,n}$ is a convex cone with dimension at most $I(m,n)$. To see
the closeness of $U_{m,n}$, we let $\mathcal{A}_k \in U_{m,n}$ with
$\mathcal{A}_k \rightarrow \mathcal{A}$. Then, for each $k \in
\mathbb{N}$, by the Carath\'{e}odory theorem, there exist $x_{k}^j \in
\mathbb{R}^n$, $j=1,\ldots,I(m,n)$, such that
\[
\mathcal{A}_k = \sum_{j=1}^{I(m,n)} (x_{k}^j)^{\otimes m}
\]
As $\mathcal{A}_k \rightarrow \mathcal{A}$, $\{\|\mathcal{A}_k\|\}_{k \in \mathbb{N}}$ is a bounded sequence. Note that
\[
\|\mathcal{A}_k\|^2 \ge  \sum_{j=1}^{I(m,n)} \sum_{i_1,\ldots,i_m=1}^{n} [(x_k^j)_{i_1}\ldots (x_k^j)_{i_m}]^2 \ge  \sum_{j=1}^{I(m,n)}\sum_{i=1}^n[(x_k^j)_i]^{2m}.
\]
So, $\{x_k^j\}_{k \in \mathbb{N}}$, $j=1,\ldots,I(m,n)$, are bounded sequences. By passing to subsequences, we can assume that $x_k^j\ \rightarrow x^j$, $j=1,\ldots,I(m,n)$.
Passing to the limit, we have
\[
 \mathcal{A} = \sum_{j=1}^{I(m,n)} (x^j)^{\otimes m} \in U_{m,n}.
\]
Thus, the conclusion follows.
\end{proof}

We now present the duality result between the PSD cone and the rank-one tensor
cone. In the case that $n=3$, Lemma 2.2 and the following result was
established in [21].
\begin{lemma}{\bf (Duality between PSD cone and rank-one tensor cone)}
 It holds that $$(U_{m,n})^{\oplus}={\rm PSD}_{m,n} \mbox{ and }{\rm PSD}_{m,n}^{\oplus}=U_{m,n}.$$
\end{lemma}
\begin{proof}
Let $\mathcal{Z} \in {\rm PSD}_{m,n}$. Then, for all $x \in \mathbb{R}^n$, $\langle \mathcal{Z}, x^{\otimes m}\rangle \ge 0$. Let $\mathcal{X} \in U_{m,n}$. Then there exist $p \in \mathbb{N}$ and $x_j \in \mathbb{R}^n$, $j=1,\ldots,p$, such that $\mathcal{X}=\sum_{j=1}^p x_j^{\otimes m}$.  It follows that
\[
\langle \mathcal{X}, \mathcal{Z}\rangle= \langle \sum_{j=1}^p x_j^{\otimes m}, \mathcal{Z}\rangle = \sum_{j=1}^p \langle \mathcal{Z}, x_j^{\otimes m}\rangle \ge 0.
\]
Thus, ${\rm PSD}_{m,n} \subseteq (U_{m,n})^{\oplus}$. To see the converse inclusion, let $\mathcal{X} \in (U_{m,n})^{\oplus}$. Note that $x^{\otimes m} \in U_{m,n}$ for all $x \in \mathbb{R}^n$. Thus,
$\langle \mathcal{X},x^{\otimes m} \rangle \ge 0$ for all $x \in \mathbb{R}^n$. This implies that $\mathcal{X} \in {\rm PSD}_{m,n}$, and so,
$(U_{m,n})^{\oplus} \subseteq {\rm PSD}_{m,n}$. Therefore, we see that $(U_{m,n})^{\oplus}={\rm PSD}_{m,n}$.

To see the second assertion, we take polars on both sides of $(U_{m,n})^{\oplus}={\rm PSD}_{m,n}$. It then follows from the double polar theorem in convex analysis (cf [44])
that
\[
{\rm PSD}_{m,n}^{\oplus}=(U_{m,n})^{\oplus\oplus}={\rm cl\, conv} (U_{m,n})=U_{m,n},
\]
where ${\rm cl \, conv}U_{m,n}$ denotes the closed and convex hull of the set $U_{m,n}$, and the last equality follows from the preceding lemma. Thus, the conclusion follows.
\end{proof}


\section{Tensor Analogy of Yuan's Alternative Theorem}
In the section, we provide an extension of Yuan's theorem of the alternative and homogeneous S-lemma to the symmetric
tensor setting. We start with the following technical proposition on hidden convexity which will be useful for our later analysis.

\begin{proposition}{\bf (Hidden Convexity)} \label{prop:1} Let $n,p \in \mathbb{N}$ and let $m$ be an even number. Let $\mathcal{F}_l$ be $m$th-order $n$-dimensional essentially non-positive symmetric tensors, $l=0,1,\ldots,p$. 
Define a set $M \subseteq \mathbb{R}^{p+1}$ by $M:=\{(\langle  \mathcal{F}_0, \mathcal{X} \rangle,\ldots,\langle  \mathcal{F}_p,  x^{\otimes m} \rangle): x \in \mathbb{R}^n\}+ {\rm int}\mathbb{R}^{p+1}_+$. Then, we have
\begin{eqnarray}\label{eq:lulu}
  M =   \{(\langle  \mathcal{F}_0, \mathcal{X} \rangle,\ldots,\langle  \mathcal{F}_p,  \mathcal{X} \rangle): \mathcal{X} \in U_{m,n}\}+ {\rm int}\mathbb{R}^{p+1}_+,
\end{eqnarray}
and $M$ is a convex cone. In particular, the following statements are equivalent:
\begin{itemize}
\item[{\rm (i)}] $(\exists x \in \mathbb{R}^n)\, (\langle \mathcal{F}_l, x^{\otimes m} \rangle < 0, \, l=0,1,\ldots,p)$;
\item[{\rm (ii)}] $(\exists \mathcal{X} \in U_{m,n})\, (\langle \mathcal{F}_l, \mathcal{X} \rangle < 0, \, l=0,1,\ldots,p)$.
\end{itemize}
\end{proposition}
 \begin{proof}
To see (\ref{eq:lulu}), we first note that
\[
 M=\{(\langle  \mathcal{F}_0, \mathcal{X} \rangle,\ldots,\langle  \mathcal{F}_p,  x^{\otimes m} \rangle): x \in \mathbb{R}^n\}+ {\rm int}\mathbb{R}^{p+1}_+ \subseteq    \{(\langle  \mathcal{F}_0, \mathcal{X} \rangle,\ldots,\langle  \mathcal{F}_p,  \mathcal{X} \rangle): \mathcal{X} \in U_{m,n}\}+ {\rm int}\mathbb{R}^{p+1}_+
\]
always holds. To get the reverse inclusion, we let $$(u_0,\ldots,u_{p}) \in    \{(\langle  \mathcal{F}_0, \mathcal{X} \rangle,\ldots,\langle  \mathcal{F}_p,  \mathcal{X} \rangle): \mathcal{X} \in U_{m,n}\}+ {\rm int}\mathbb{R}^{p+1}_+ \, .$$
Then, there exist $\mathcal{X} \in U_{m,n}$ such that
\begin{equation}\label{eq:dreamlulu}
\langle  \mathcal{F}_l, \mathcal{X} \rangle < u_l,\ l=0,1,\ldots,p.
\end{equation}
As $\mathcal{X} \in U_{m,n}$ and $U_{m,n}$ is a closed and convex cone with dimension at most $I(m,n)$, there exist $u^j \in \mathbb{R}^n$ such that
\begin{equation}\label{eq:star}
\mathcal{X} = \sum_{j=1}^{I(m,n)} (u^j)^{\otimes m}.
\end{equation}
Define $\bar x \in \mathbb{R}^n$ by $$\bar x=(\sqrt[m]{\mathcal{X}_{1,\ldots,1}},\ldots,\sqrt[m]{\mathcal{X}_{n,\ldots,n}})=\left(\, \sqrt[m]{\sum_{j=1}^{I(m,n)}(u^j)_1^m},\ldots,\sqrt[m]{\sum_{j=1}^{I(m,n)}(u^j)_n^m}\, \right).$$
We now show that $\langle \mathcal{F}_l, \bar x^{\otimes m} \rangle \le \langle \mathcal{F}_l, \mathcal{X} \rangle$ for all  $l=0,1,\ldots,p$. To see this, let $$I=\{(i_1,\ldots,i_m): i_1=\cdots=i_m\}.$$  Then, for each $l=0,1,\ldots,p$, we have
\begin{eqnarray}\label{eq:uu}
 \langle \mathcal{F}_l, \bar x^{\otimes m} \rangle
 & = & \sum_{i_1,\ldots,i_m=1}^n (\mathcal{F}_l)_{i_1 \ldots i_m} \sqrt[m]{\sum_{j=1}^{I(m,n)}(u^j)_{i_1}^m} \ldots \sqrt[m]{\sum_{j=1}^{I(m,n)}(u^j)_{i_m}^m}  \nonumber \\
 & = & [\sum_{i=1}^n (\mathcal{F}_l)_{i \ldots i} \, \big(\sum_{j=1}^{I(m,n)}(u^j)_{i}^m \big) ] + \sum_{(i_1,\ldots,i_m) \notin I}(\mathcal{F}_l)_{i_1 \ldots i_m} \sqrt[m]{\sum_{j=1}^{I(m,n)}(u^j)_{i_1}^m} \ldots \sqrt[m]{\sum_{j=1}^{I(m,n)}(u^j)_{i_m}^m} \nonumber \\
 & = & [\sum_{i=1}^n (\mathcal{F}_l)_{i \ldots i} \, \mathcal{X}_{i \ldots i}] + \sum_{(i_1,\ldots,i_m) \notin I}(\mathcal{F}_l)_{i_1 \ldots i_m} \sqrt[m]{\sum_{j=1}^{I(m,n)}(u^j)_{i_1}^m} \cdots \sqrt[m]{\sum_{j=1}^{I(m,n)}(u^j)_{i_m}^m}\ .
\end{eqnarray}
Recall the following generalized H\"{o}lder inequality (cf [45]): for $q \in \mathbb{N}$ and $a_{kj} \ge 0$, $k=1,\ldots,m$ and $j=1,\ldots,q$
\[
\prod_{k=1}^m \sum_{j=1}^q a_{kj} \ge \left(\, \sum_{j=1}^q\sqrt[m]{\prod_{k=1}^ma_{kj}} \, \right)^m.
\]
Applying this inequality with $a_{kj}=(u^j)_{i_k}^m \ge 0$ and $q=I(m,n)$, we have
\[
\prod_{k=1}^m \sum_{j=1}^{I(m,n)} (u^j)_{i_k}^m \ge \left(\, \sum_{j=1}^{I(m,n)}\sqrt[m]{\prod_{k=1}^m (u^j)_{i_k}^m} \, \right)^m=\left(\, \sum_{j=1}^{I(m,n)} {\prod_{k=1}^m |(u^j)_{i_k}}| \, \right)^m.
\]
and so,
\[
\prod_{k=1}^m \sqrt[m]{\sum_{j=1}^{I(m,n)} (u^j)_{i_k}^m} \ge  \sum_{j=1}^{I(m,n)} {\prod_{k=1}^m (u^j)_{i_k}}. \]
As $\mathcal{F}_l$ are essentially non-positive, for each $l=0,1,\ldots,p$, $(\mathcal{F}_l)_{i_1 \ldots i_m} \le 0$ for all $(i_1,\ldots,i_m) \notin I$.
This together with (\ref{eq:uu}) implies that
\begin{eqnarray*}
\langle \mathcal{F}_l, \bar x^{\otimes m} \rangle  & \le & \sum_{i=1}^n (\mathcal{F}_l)_{i \ldots i} \, \mathcal{X}_{i \ldots i} + \sum_{(i_1,\ldots,i_m) \notin I}(\mathcal{F}_l)_{i_1 \ldots i_m} \sum_{j=1}^{I(m,n)} (u^j)_{i_1}\cdots(u^j)_{i_m} \\
 & = & \sum_{i=1}^n (\mathcal{F}_l)_{i \ldots i} \, \mathcal{X}_{i \ldots i}+ \sum_{(i_1,\ldots,i_m) \notin I}(\mathcal{F}_l)_{i_1 \ldots i_m} \mathcal{X}_{i_1 \ldots i_m} \\
 & = & \langle \mathcal{F}_l, \mathcal{X} \rangle ,
\end{eqnarray*}
where the first equality follows from (\ref{eq:star}).
Thus, from (\ref{eq:dreamlulu}), we have $\langle \mathcal{F}_l, \bar x^{\otimes m} \rangle < u_l, \ l=0,1,\ldots,p.$ So,
$(u_0,\ldots,u_{p}) \in  M=\{(\langle  \mathcal{F}_0, \mathcal{X} \rangle,\ldots,\langle  \mathcal{F}_p,  x^{\otimes m} \rangle): x \in \mathbb{R}^n\}+ {\rm int}\mathbb{R}^{p+1}_+ $, and hence
(\ref{eq:lulu}) holds.  From (\ref{eq:lulu}),  we see that $M$ is clearly a convex cone. Finally, the equivalence between the statements {\rm (i)} and {\rm (ii)} follows immediately by (\ref{eq:lulu}).
%
%
 \end{proof}
\begin{remark}{\bf (A useful inequality)}\label{remark:2.1}
The proof of the preceding proposition gives us the following useful inequality: Let $\mathcal{X} \in  U_{m,n}$ and $\mathcal{F} \in E_{m,n}$. Define $\bar x=(\sqrt[m]{\mathcal{X}_{1,\ldots,1}},\ldots,\sqrt[m]{\mathcal{X}_{n,\ldots,n}}).$ Then, we have
\[
\langle \mathcal{F}, \bar x^{\otimes m} \rangle \le \langle \mathcal{F}, \mathcal{X} \rangle.
\]
\end{remark}
  Let $P=(P_{ij})$ be an $n \times n$  real matrix. Define $\mathcal{B}= P^m  \mathcal{A}$ as an $m$th-order $n$-dimensional  tensor where its entries are given by
 \[
 \mathcal{B}_{i_1 \cdots i_m}=\sum_{j_1 \cdots j_m=1}^n P_{i_1 j_1} \cdots P_{i_m j_m} \mathcal{A}_{j_1 \cdots j_m} .
 \]

 \begin{lemma}\label{lemma:1}
 For a symmetric  $m$th-order $n$-dimensional tensor $\mathcal{A}$ and an $(n \times n)$ matrix $P$, we have
$\langle \mathcal{A}, (P^Tx)^{\otimes m}\rangle = \langle P^m \mathcal{A}, x^{\otimes m}\rangle \mbox{ for all } x \in \mathbb{R}^n$.
\begin{proof}
From the definition, we have
{\small \begin{eqnarray*}
\langle \mathcal{A}, (P^Tx)^{\otimes m}\rangle =  \sum_{i_1 \cdots i_m=1}^n \mathcal{A}_{i_1 \cdots i_m} (P^Tx)_{i_1}\cdots (P^T x)_{i_m}&=& \sum_{i_1 \cdots i_m=1}^n \mathcal{A}_{i_1 \cdots i_m} (\sum_{j_1=1}^n P_{j_1 i_1}x_{j_1})\cdots (\sum_{j_m=1}^n P_{j_m i_m}x_{j_m}) \\
& = & \sum_{i_1 \cdots i_m=1}^n \mathcal{A}_{i_1 \cdots i_m} \sum_{j_1 \cdots j_m=1}^n \bigg(P_{j_1 i_1}x_{j_1} \cdots P_{j_m i_m}x_{j_m}\bigg) \\
& = & \sum_{i_1 \cdots i_m=1}^n \sum_{j_1 \cdots j_m=1}^n \mathcal{A}_{i_1 \cdots i_m}  \bigg(P_{j_1 i_1}x_{j_1} \cdots P_{j_m i_m}x_{j_m}\bigg) \\
& = & \sum_{j_1 \cdots j_m=1}^n\sum_{i_1 \cdots i_m=1}^n  P_{j_1 i_1}\cdots P_{j_m i_m} \mathcal{A}_{i_1 \cdots i_m}  x_{j_1} \cdots x_{j_m}\ .
\end{eqnarray*}}
Note that
\[
(P^m \mathcal{A})_{j_1 \cdots j_m}=\sum_{i_1 \cdots i_m=1}^n  P_{j_1 i_1}\cdots P_{j_m i_m} \mathcal{A}_{i_1 \cdots i_m}.
\]
It follows that
\begin{eqnarray*}
\langle \mathcal{A}, (P^Tx)^{\otimes m}\rangle & = & \sum_{j_1 \cdots j_m=1}^n\sum_{i_1 \cdots i_m=1}^n  P_{j_1 i_1}\cdots P_{j_m i_m} \mathcal{A}_{i_1 \cdots i_m}  x_{j_1} \cdots x_{j_m} \\
& = &  \sum_{j_1 \cdots j_m=1}^n (P^m \mathcal{A})_{j_1 \cdots j_m} x_{j_1} \cdots x_{j_m}\\
& = &  \langle P^m \mathcal{A}, x^{\otimes m}\rangle.
\end{eqnarray*}
Thus the conclusion follows.
\end{proof}
 \end{lemma}

 We are now ready to state the extension of Yuan's theorem of the alternative in symmetric tensor setting.
 \begin{theorem}{\bf (Tensor Analogy of Yuan's Alternative Theorem)} \label{th:1}
 Let $n,p \in \mathbb{N}$ and let $m$ be an even number. Let $\mathcal{F}_l$, $l=0,1,\ldots,p$, be $m$th-order $n$-dimensional symmetric tensors. Suppose that there exists a nonsingular $(n \times n)$ matrix $P$
such that $P^m\mathcal{F}_l$, $l=0,1,\ldots,p$, are all essentially nonpositive tensors.  Then, one and exactly one of the following statements holds:
\begin{itemize}
\item[{\rm (i)}] $(\exists x \in \mathbb{R}^n)\, (\langle \mathcal{F}_l, x^{\otimes m} \rangle < 0, \, l=0,1,\ldots,p)$;
\item[{\rm (ii)}] $\displaystyle (\exists \lambda_l \ge 0, l=0,1,\ldots,p, \sum_{l=0}^p \lambda_l=1) \, (\sum_{l=0}^{p}\lambda_l \mathcal{F}_l \in {\rm SOS}_{m,n})$,
\end{itemize}
where ${\rm SOS}_{m,n}$ is the $m$th-order $n$-dimensional sum-of-squares tensor cone and $x^{\otimes m}$ is the $m$th-order $n$-dimensional rank-one tensor induced by $x$.
 \end{theorem}
\begin{proof}

[${\rm (ii)} \Rightarrow {\rm Not} {\rm (i)}$] Suppose that statement {\rm (ii)} holds. Then, there exist $\lambda_l \ge 0$, $l=0,1,\ldots,p$, with $\sum_{l=0}^p\lambda_l =1$ such that
\[
\sum_{l=0}^{p}\lambda_l \mathcal{F}_l \in {\rm SOS}_{m,n}.
\] We now establish that {\rm (i)} must fail by using the method of contradiction. Suppose that ${\rm (i)}$ holds.
Then, there exists  $u \in \mathbb{R}^n$ be such that $ \langle \mathcal{F}_l ,u^{\otimes m}\rangle < 0$, $l=0,1,\ldots,p$. It follows that
\[
0 \le \langle \sum_{l=0}^{p}\lambda_l \mathcal{F}_l, u^{\otimes m}\rangle  \le \max_{0 \le l \le p}\langle \mathcal{F}_l ,u^{\otimes m}\rangle <0.
\]
This is impossible, and so {\rm (i)} must fail.

[${\rm Not} {\rm (i)} \Rightarrow {\rm (ii)}$] Suppose that {\rm (i)} fails. Then, the following system has no solution:
$$(\exists x \in \mathbb{R}^n)\, (\langle \mathcal{F}_l, x^{\otimes m} \rangle < 0, \, l=0,1,\ldots,p).$$ Letting $x=P^Ty$, this implies that the following system  has no solution:
$$(\exists y \in \mathbb{R}^n)\, (\langle \mathcal{F}_l, (P^Ty)^{\otimes m} \rangle < 0, \, l=0,1,\ldots,p).$$
Note from the preceding lemma that $\langle \mathcal{F}_l, (P^Ty)^{\otimes m}\rangle =\langle P^m \mathcal{F}_l, y^{\otimes m}\rangle$.
This together with the equivalence between statements {\rm (i)} and {\rm (ii)} in Proposition \ref{prop:1} implies that  the following system also has no solution
\[
(\exists \mathcal{X} \in U_{m,n})(\langle P^m \mathcal{F}_l, \mathcal{X} \rangle < 0, \, l=0,1,\ldots,p).
\]
This implies that
$(0,\ldots,0) \notin \{(\langle P^m \mathcal{F}_0, \mathcal{X} \rangle,\ldots,\langle P^m \mathcal{F}_p, \mathcal{X} \rangle):\mathcal{X} \in U_{m,n}\}+ {\rm int}\mathbb{R}^{p+1}_+$.
As $U_{m,n}$ is a convex cone, $\{(\langle P^m \mathcal{F}_0, \mathcal{X} \rangle,\ldots,\langle P^m \mathcal{F}_p, \mathcal{X} \rangle):\mathcal{X} \in U_{m,n}\}$ is also a convex cone, and so,
\[
C:=\{(\langle P^m \mathcal{F}_0, \mathcal{X} \rangle,\ldots,\langle P^m \mathcal{F}_p, \mathcal{X} \rangle):\mathcal{X} \in U_{m,n}\}+ {\rm int}\mathbb{R}^{p+1}_+
\]
is a convex cone. Then, the standard separation theorem (cf [44, Theorem 1.1.3]) implies that there exists $(\mu_0,\ldots,\mu_p) \in \mathbb{R}^{p+1}\backslash\{0\}$ such that
\[
0 \le \sum_{l=0}^{p}\mu_l a_l \  \mbox{ for all } \  (a_0,a_1,\ldots,a_p) \in C.
\]
As $C+{\rm int}\mathbb{R}^{p+1}_+ \subseteq C$, it follows that $\mu_l \ge 0$, $l=0,1,\ldots,p$. So, $(\mu_0,\ldots,\mu_p) \in \mathbb{R}^{p+1}_+\backslash\{0\}$ and hence $\sum_{l=0}^p\mu_l >0$. Let $\lambda_l=\frac{\mu_l}{\sum_{l=0}^p\mu_l}\ge 0$. Then, $\sum_{l=0}^p\lambda_l=1$ and
\[
\sum_{l=0}^{p}\lambda_l a_l \ge 0 \  \mbox{ for all } \  (a_0,a_1,\ldots,a_p) \in C
\]
In particular, this shows that, for each $\epsilon>0$,
\[
\sum_{l=0}^{p}\lambda_l (\langle P^m\mathcal{F}_l, \mathcal{X} \rangle)+\epsilon=\sum_{l=0}^{p}\lambda_l (\langle P^m \mathcal{F}_l, \mathcal{X}  \rangle+\epsilon) \ge 0 \mbox{ for all } \mathcal{X} \in  U_{m,n}.
\]
Let $\epsilon \rightarrow 0$. This implies that
\[
\sum_{l=0}^{p}\lambda_l \langle P^m \mathcal{F}_l, \mathcal{X} \rangle  \ge 0 \mbox{ for all } \mathcal{X} \in  U_{m,n}.
\]
In other words,
\[
\sum_{l=0}^{p}\lambda_l \, P^m \mathcal{F}_l \in (U_{m,n})^{\oplus} ={\rm PSD}_{m,n}.
\]
To finish the proof, we only need to show that $\displaystyle \sum_{l=0}^{p}\lambda_l \, \mathcal{F}_l \in {\rm SOS}_{m,n}$. To see this, note from our assumption that
$\displaystyle \sum_{l=0}^{p}\lambda_l \, P^m \mathcal{F}_l \in E_{m,n}$. It follows that $\displaystyle \sum_{l=0}^{p}\lambda_l \, P^m \mathcal{F}_l \in {\rm PSD}_{m,n} \cap E_{m,n}$. Then, Proposition \ref{prop:0.1} gives us that $\displaystyle \sum_{l=0}^{p}\lambda_l \, P^m \mathcal{F}_l \in {\rm SOS}_{m,n}.$ So, $\sigma(x):=\langle \sum_{l=0}^{p}\lambda_l \, P^m \mathcal{F}_l,x^{\otimes m}\rangle$ is a sum-of-squares polynomial on $\mathbb{R}^n$ with degree $m$. This together with Lemma \ref{lemma:1} implies that for all $z \in \mathbb{R}^n$
\[
 \sum_{l=0}^{p}  \langle \, \lambda_l\mathcal{F}_l,(P^Tz)^{\otimes m}\rangle = \sum_{l=0}^{p}\lambda_l  \langle \, P^m \mathcal{F}_l,z^{\otimes m}\rangle =\langle \sum_{l=0}^{p}\lambda_l \, P^m \mathcal{F}_l,z^{\otimes m}\rangle=\sigma(z).
\]
So, for all $x \in \mathbb{R}^n$
\[
\sum_{l=0}^{p}  \langle \, \lambda_l\mathcal{F}_l,x^{\otimes m}\rangle =\sigma((P^T)^{-1}x)
\]
is also a sum-of-squares polynomial on $\mathbb{R}^n$ with degree $m$. Thus, $\displaystyle \sum_{l=0}^{p}\lambda_l \, \mathcal{F}_l \in {\rm SOS}_{m,n}$, and hence the conclusion follows.
\end{proof}

In the matrix case, Theorem \ref{th:1} reduces to the following theorem of the alternative presented in [10] (see also [11]).
\begin{corollary}\label{cor:pp} {\bf (Matrix Cases)}
Let $A_0, A_1,\ldots,A_p$, $p \in \mathbb{N}$ be symmetric $(n
 \times n)$ matrices. Suppose that there exists a
nonsingular $(n \times n)$ matrix $Q$ such that $Q^TA_0Q, Q^TA_1Q,  \ldots,
Q^TA_pQ$ are all matrices with   non-positive off-diagonal elements.
Then  exactly one of the following statements holds:
\begin{itemize}
\item[{\rm (i)}] there exists $x \in \mathbb{R}^n$ such that   $x^TA_lx<0$, $l=0,1,\ldots,p$;
\item[{\rm (ii)}] $\displaystyle (\exists \lambda_l \ge 0, l=0,1,\ldots,p, \sum_{l=0}^p \lambda_l=1) \, (\sum_{l=0}^{p}\lambda_l A_l$ is positive semidefinite$)$. 
\end{itemize}
\end{corollary}
\begin{proof}
In the special case when $m=2$ (and so, $\mathcal{F}_l=F_l$ are  $(n \times n)$ symmetric matrices), we have
 $\langle \mathcal{F}_l, x^{\otimes m} \rangle=x^T F_l x$ and ${\rm SOS}_{m,n}$ collapses to the positive semi-definite matrix cone. So, the conclusion follows from the preceding theorem by letting the order $m=2$.
\end{proof}


We note that, in Corollary \ref{cor:pp}, the assumption ``there exists a
nonsingular $(n \times n)$ matrix $Q$ such that $Q^TA_0Q, Q^TA_1Q,  \ldots,
Q^TA_pQ$ are all matrices with   non-positive off-diagonal elements'' is superfluous when only two quadratic functions are involved (that is, $p=1$). This was explained in [10, Remark 2.3]. In this case,
Corollary \ref{cor:pp} reduces to Yuan's theorem of the alternative. Therefore, Theorem \ref{th:1} can  be regarded as
an extension of Yuan's theorem of the alternative to the symmetric tensor setting.

However, unlike the matrix cases, if the condition ``there exists a nonsingular $(n \times n)$ matrix $P$ such that $P^m\mathcal{F}_l$, $l=0,1,\ldots,p$, are all essentially nonpositive tensors`` is dropped,
the above tensor analogy of Yuan's theorem of the alternative can fail even in the case $p=1$.  We illustrate this fact by the following example.
\begin{example}
Let $f_M$ be the homogeneous
Motzkin polynomial, that is, $$f_M(x_1,x_2,x_3)=x_3^6 + x_1^2x_2^4 + x_1^4x_2^2-3x_1^2x_2^2x_3^2, \ x=(x_1,x_2,x_3) \in \mathbb{R}^3.$$ Let $f_0,f_1$ be polynomials with degree $6$ on $\mathbb{R}^4$ defined by $$f_0(x_1,x_2,x_3,x_4)=f_M(x_1,x_2,x_3) \mbox{ and }  f_1(x_1,x_2,x_3,x_4)=x_1^6+x_2^6+x_3^6-x_4^6.$$
Let $\mathcal{F}_i$  be the symmetric tensors
associated to $f_i$, $i=1,2$, in the sense that $f_i(x)=\langle \mathcal{F}_i, x^{\otimes 6}\rangle$, for all $x \in \mathbb{R}^4$. As the homogeneous Motzkin polynomial always takes non-negative value, one cannot find $x \in \mathbb{R}^4$ such that $\langle \mathcal{F}_i, x^{\otimes m} \rangle < 0, \, i=0,1$. This implies that statement {\rm (1)} in Theorem \ref{th:1} fails. We now see that statement {\rm (2)} in Theorem \ref{th:1} also fails. Suppose on the contrary that there exist $\lambda_0,\lambda_1 \ge 0$ with $\lambda_0+\lambda_1=1$ such that
$\lambda_0\mathcal{F}_0+\lambda_1 \mathcal{F}_1 \in {\rm SOS}_{6,4}$.
Note that $f_0$ does not depend on $x_4$ and $f_1(x_1,x_2,x_3,x_4) \rightarrow -\infty$ as $x_4\rightarrow \infty$ for fixed $x_1,x_2,x_3$. It follows that $\lambda_1=0$ $($and so, $\lambda_0=1$$)$. Hence, $\mathcal{F}_0  \in {\rm SOS}_{6,4}$. This contradicts the fact that the homogeneous
Motzkin polynomial is not a sum-of-squares polynomial. Therefore, for this example, statements {\rm (i)} and {\rm (ii)} in Theorem \ref{th:1} both  fail.
\end{example}

As a consequence, we now provide an extension of the homogeneous S-lemma as follows.
\begin{corollary}{\bf (Tensor Analogy of Homogeneous S-lemma)}\label{cor:S-lemma}
   Let $n,p \in \mathbb{N}$ and let $m$ be an even number. Let $\mathcal{F}_l$, $l=0,1,\ldots,p$, be $m$th-order $n$-dimensional symmetric tensors. Suppose that there exists a nonsingular matrix $P$
such that $P^m\mathcal{F}_l$, $l=0,1,\ldots,p$, are all essentially nonpositive tensors. Suppose that there exists $x_0 \in \mathbb{R}^n$ such that
$\langle \mathcal{F}_l, x_0^{\otimes m} \rangle < 0, \, l=1,\ldots,p$. Then, the following statements are equivalent:
\begin{itemize}
\item[{\rm (i)}] $\langle \mathcal{F}_l, x^{\otimes m} \rangle \le 0, \, l=1,\ldots,p \ \Rightarrow \ \langle \mathcal{F}_0, x^{\otimes m} \rangle \ge 0$;
\item[{\rm (ii)}] $\displaystyle (\exists \lambda_l \ge 0, l=1,\ldots,p) \, (\mathcal{F}_0 +\sum_{l=1}^{p}\lambda_l \mathcal{F}_l \in {\rm SOS}_{m,n})$.
\end{itemize}
\end{corollary}
\begin{proof}

[${\rm (ii)} \Rightarrow {\rm (i)}$] Suppose that statement {\rm (ii)} holds. Then, there exist $\lambda_l \ge 0$, $l=1,\ldots,p$ such that
\[
\mathcal{F}_0+\sum_{l=1}^{p}\lambda_l \mathcal{F}_l \in {\rm SOS}_{m,n} \subseteq {\rm PSD}_{m,n}.
\]
Let $x \in \mathbb{R}^n$ such that $\langle \mathcal{F}_l, x^{\otimes m} \rangle \le 0, \, l=1,\ldots,p$. Then,
\[
0 \le \langle \mathcal{F}_0+\sum_{l=1}^{p}\lambda_l \mathcal{F}_l,x^{\otimes m}\rangle =\langle \mathcal{F}_0,x^{\otimes m}\rangle +\sum_{l=1}^{p}\lambda_l \langle \mathcal{F}_l,x^{\otimes m}\rangle \le \langle \mathcal{F}_0,x^{\otimes m}\rangle.
\]
Thus, {\rm (i)} follows.

[${\rm (i)} \Rightarrow {\rm (ii)}$] Suppose that {\rm (i)} holds. Then, the following inequality system has no solution
$$(\exists x \in \mathbb{R}^n)\, (\langle \mathcal{F}_l, x^{\otimes m} \rangle < 0, \, l=0,1,\ldots,p).$$ Then, the preceding alternative
theorem implies that there exist $\bar \lambda_l \ge 0, l=0,1,\ldots,p$, with  $\displaystyle \sum_{l=0}^p \bar \lambda_l=1$ such that $$\sum_{l=0}^{p}\bar \lambda_l \mathcal{F}_l=\mathcal{A} \in {\rm SOS}_{m,n}. $$
 We now observe that $\bar \lambda_0>0$ (Otherwise, $\bar \lambda_0=0$ and so, $\sum_{l=1}^p \bar \lambda_l=1$ and $\sum_{l=1}^{p}\bar \lambda_l \mathcal{F}_l \in {\rm SOS}_{m,n}.$ This gives us that
\[
0 \le \langle \sum_{l=1}^{p}\bar \lambda_l \mathcal{F}_l,x_0^{\otimes m}\rangle=\sum_{l=1}^{p}\bar \lambda_l \langle \mathcal{F}_l,x_0^{\otimes m}\rangle \le \max_{1 \le l \le p}\langle \mathcal{F}_l,x_0^{\otimes m}\rangle<0,
\]
which is impossible).  Let $\lambda_l=\bar \lambda_l/ \bar \lambda_0$, $l=1,\ldots,p$. Then, we have
$$\mathcal{F}_0 +\sum_{l=1}^{p}\lambda_l \mathcal{F}_l= (\bar \lambda_0)^{-1} \mathcal{A} \in {\rm SOS}_{m,n}.$$ Thus, the conclusion follows.
\end{proof}
\begin{remark}
{\rm Similar to the Yuan's theorem of the alternative, in the case when $m=2$ and $p=1$ (that is, inequality system involving two homogeneous quadratic functions), the above corollary collapses to the well-known homogeneous S-lemma (cf. [4])}.
\end{remark}

\section{Application: Polynomial Optimization with Essentially Nonpositive Coefficients}
In this section, as an application of our theorem of the alternative, we establish an exact conic programming relaxation result for polynomial optimization problems with essentially nonpositive coefficients. To do this, we first introduce the definition of polynomials with essentially nonpositive coefficients.
\begin{definition}{\bf (Polynomials with essentially nonpositive coefficients)}
Let $f$ be a polynomial on $\mathbb{R}^n$ with degree $m$. Let $r=f(0)$ be the constant term of $f$ and let $f_{m,i}$ be the coefficient associated with $x_i^{m}$.
Recall that \begin{equation}\label{eq:Omega_f}\Omega_f=\{\alpha=(\alpha_1,\ldots,\alpha_n) \in (\mathbb{N} \cup \{0\})^n: f_{\alpha} \neq 0 \mbox{ and } \alpha \neq m e_i, \ i=1,\ldots,n\},\end{equation}
where $e_i$ be the vector whose $i$th component is one and all the other components are zero. We note that $f$ can be written as
\[
f(x)=\sum_{i=1}^n f_{m,i} x_i^{m}+\sum_{\alpha \in \Omega_f \backslash \{0\}}f_{\alpha}x^{\alpha}+r.
\]
 We say $f$ has essentially nonpositive coefficients if $f_{\alpha} \le 0$ for all $\alpha\in \Omega_f \backslash \{0\}$.
\end{definition}
 Let $n \in \mathbb{N}$ and let $m$ be an even number. Consider the following nonconvex polynomial optimization problem with essentially nonpositive coefficients:
\begin{eqnarray*}
(P) \ \ \ \ \ \ \ \ \ \ \ \ \ \ \ \ \ \ \ \ \ \ \ & \displaystyle \min_{x \in \mathbb{R}^n}\{f_0(x):f_l(x) \le 0, \, l=1,\ldots,p\}.\ \ \ \ \ \ \ \ \ \ \ \ \ \ \ \ \ \ \ \ \ \ \ \ \ \ \ \ \  \ \ \ \ \ \ \ \ \ \ \ \ \ \ \ \ \ \ \ \ \ \ \ \ \ \ \ \ \ \ \ \ \ \ \ \ \ \ \ \ \ \ \ \ \ \ \ \ \ \ \ \
\end{eqnarray*}
where $f_l$, $l=0,1,\ldots,p$, are polynomials on $\mathbb{R}^n$ with essentially nonpositive coefficients and degree $m$. We use $\min(P)$ to denote the optimal value of problem (P). Throughout this section, we always assume that the feasible set of (P) is nonempty.

Below, we first establish that the optimal value of problem (P) can be found by a conic programming problem and the optimal solution of (P) can be recovered by a solution of the corresponding conic programming problem. To do this, we introduce the canonical homogenization of a polynomial and a conic programming problem as follows.

Consider the following conic programming problem
 \begin{eqnarray*}
(CP) \ \ \ \ \ \ \ \ \ \ \ \ \ \ \ \ \ \ \ \ \ \ \  &\min \{\langle \tilde{\mathcal{F}}_0, \mathcal{X} \rangle: \langle \tilde{\mathcal{F}}_l, \mathcal{X} \rangle \le 0, \ l=1,\ldots,p,\, \mathcal{X}_{n+1 \ldots n+1}=1,\, \mathcal{X} \in U_{m,n+1}\}. \ \ \ \ \ \ \ \ \ \ \ \ \ \ \ \ \ \ \ \ \ \ \ \ \ \ \ \ \  \ \ \ \ \ \ \ \ \ \ \ \ \ \ \ \ \ \ \ \ \ \ \ \ \ \ \ \ \ \ \ \ \ \ \ \ \ \ \ \ \ \ \ \ \ \ \ \ \ \ \ \
\end{eqnarray*}
where each $\tilde{\mathcal{F}}_l$, $l=0,1,\ldots,p$, is the symmetric tensor associated with the canonical homogenization of  $\tilde{f}_l$, that is, $\tilde{f}_l(\tilde{x})=\langle \tilde{\mathcal{F}}_l, \tilde{x}^{\otimes m}\rangle$ for any $\tilde{x}=(x^T,t)^T \in \mathbb{R}^{n+1}$.
\begin{lemma}
 Let $n \in \mathbb{N}$ and let $m$ be an even number. Let $f$ be a polynomial on $\mathbb{R}^n$ with essentially nonpositive  coefficients and
 degree $m$, and let $\tilde{f}$ be the canonical homogenization of $f$. Let $\tilde{\mathcal{F}}$ be the symmetric tensor associated with the canonical homogenization of  $\tilde{f}$, that is, $\tilde{f}(\tilde{x})=\langle \tilde{\mathcal{F}}, \tilde{x}^{\otimes m}\rangle$ for any $\tilde{x}=(x^T,t)^T \in \mathbb{R}^{n+1}.$ Then, $\tilde{\mathcal{F}}$ is an essentially nonpositive tensor.
\end{lemma}
\begin{proof}
For each real polynomial $f$ with essentially nonpositive  coefficients, we can decompose it as
\[
f(x)=\sum_{i=1}^n f_{m,i} x_i^{m}+\sum_{\alpha \in \Omega_f \backslash \{0\}}f_{\alpha}\, x^{\alpha}+r,
\]
where $r=f(0)$ and $f_{\alpha} \le 0$ for all $\alpha \in \Omega_f \backslash \{0\}$.
Its canonical homogenization can be written as
\[
\tilde{f}(x,t)=\sum_{i=1}^n f_{m,i} x_i^{m}+\sum_{\alpha \in \Omega_f \backslash \{0\}}f_{\alpha}\, x^{\alpha}t^{m-|\alpha|}+r\, t^m \mbox{ for all } (x^T,t)^T \in \mathbb{R}^{n+1}.
\]
Recall that $\tilde{\mathcal{F}}$ is the symmetric tensor associated with the canonical homogenization of  $\tilde{f}$, that is, for all
$\tilde{x}=(x^T,t)^T \in \mathbb{R}^{n+1}$,
$$\tilde{f}(\tilde{x})=\langle \tilde{\mathcal{F}}, \tilde{x}^{\otimes m}\rangle=\sum_{i_1 \cdots i_m=1}^{n+1} \tilde{\mathcal{F}}_{i_1,\ldots,
i_m}\tilde{x}_{i_1}\ldots \tilde{x}_{i_m}  .$$
Note that each $m$-th order $(n+1)$-dimensional symmetric tensor uniquely corresponds a degree $m$ homogeneous polynomial on $\mathbb{R}^{n+1}$. As $f_{\alpha} \le 0$ for all $\alpha \in \Omega_f \backslash \{0\}$, it follows that
\[
\tilde{\mathcal{F}}_{i_1,\ldots,
i_m} \le 0 \mbox{ for all } (i_1,\ldots,i_m) \in I,
\]
where $I:=\{(i,i,\ldots,i) \in \mathbb{N}^m: 1 \le i \le n+1\}.$
Therefore, $\tilde{\mathcal{F}}$ is essentially nonpositive.
\end{proof}
\begin{theorem}{\bf (Exact Solutions via Conic Programs)}
 Let $n \in \mathbb{N}$ and let $m$ be an even number. Let $\tilde{\mathcal{F}_l}$ be the $m$th-order $(n+1)$-dimensional symmetric tensor associated with the canonical homogenization of  $\tilde{f}_l$, that is, $\tilde{f}_l(\tilde{x})=\langle \tilde{\mathcal{F}}_l, \tilde{x}^{\otimes m}\rangle$ for any $\tilde{x}=(x^T,t)^T \in \mathbb{R}^{n+1}.$ Consider the  nonconvex polynomial optimization problem with essentially nonpositive coefficients (P) and its associated conic relaxation problem (CP).  Then, we have $\min(P) = \min(CP).$
Moreover, for any solution $\bar{\mathcal{X}}$ of (CP), $$\bar x:=(\sqrt[m]{\bar{\mathcal{X}}_{1,\ldots,1}},\ldots,\sqrt[m]{\bar{\mathcal{X}}_{n,\ldots,n}}) \in \mathbb{R}^n$$ is a solution of (P).
\end{theorem}
\begin{proof}
Let $\tilde{f}_l$ be the canonical homogenization of $f_l$, $l=0,1,\ldots,p$. Note that $\tilde{f}_l(x,1)=f_l(x)$ for all $x \in \mathbb{R}^n$ and $l=0,1,\ldots,p$. We first see that
\begin{eqnarray*}
\min(P) &= &  \min_{\tilde{x}=(x^T,t)^T \in \mathbb{R}^{n+1}}\{\tilde{f}_0(x,t): \tilde{f}_l(x,t) \le 0, \, l=1,\ldots,p, \, t=1\} \\
& = & \min_{\tilde{x}=(x^T,t)^T \in \mathbb{R}^{n+1} }\{\langle \tilde{\mathcal{F}}_0,\tilde{x}^{\otimes m}\rangle: \langle \tilde{\mathcal{F}}_l,\tilde{x}^{\otimes m}\rangle\le 0,\, l=1,\ldots,p, \, t=1\} \\
& \ge & \min \{\langle \tilde{\mathcal{F}}_0, \mathcal{X} \rangle: \langle \tilde{\mathcal{F}}_l, \mathcal{X} \rangle \le 0, \ l=1,\ldots,p, \, \mathcal{X}_{n+1,\ldots,n+1}=1,\, \mathcal{X} \in U_{m,n+1}\} \\
& =& \min(CP),
\end{eqnarray*}
where the inequality follows as
$\{\tilde{x}^{\otimes m}: \tilde{x}=(x^T,1)^T \in \mathbb{R}^{n+1}\} \subseteq \{\mathcal{X} \in U_{m,n+1}: \mathcal{X}_{n+1,\ldots,n+1}=1\}$.

On the other hand, let $\mathcal{X} \in U_{m,n+1}$ with $\langle \tilde{\mathcal{F}}_l, \mathcal{X} \rangle \le 0, \ l=1,\ldots,p$ and $\mathcal{X}_{n+1,\ldots,n+1}=1$.  Define $x=(\sqrt[m]{\mathcal{X}_{1,\ldots,1}},\ldots,\sqrt[m]{\mathcal{X}_{n,\ldots,n}})$ and $\tilde{x}=(x^T,1)^T$. As $\tilde{\mathcal{F}}_l$, $l=0,1,\ldots,p$, are all essentially nonpositive tensors, then Remark \ref{remark:2.1} implies that
\[
f_l(x)=\tilde{f}_l(\tilde{x})=\langle \tilde{\mathcal{F}}_l, \tilde{x}^{\otimes m} \rangle \le \langle \tilde{\mathcal{F}}_l, \mathcal{X} \rangle.
\]
This implies that, $f_l(x) \le 0$ for each $l=1,\ldots,p$ (and so, $x$ is feasible for (P)), and $f_0(x) \le \langle \tilde{\mathcal{F}}_0, \mathcal{X} \rangle.$ So,
$\min(P) \le \min(CP)$.
Thus, we see that $\min(P)=\min(CP).$

To see the last assertion, let $\bar{\mathcal{X}}$ be a solution of (CP) and let $$\bar x:=(\sqrt[m]{\bar{\mathcal{X}}_{1,\ldots,1}},\ldots,\sqrt[m]{\bar{\mathcal{X}}_{n,\ldots,n}}) \in \mathbb{R}^n.$$
Then, using similar argument as before, we have $f_l(\bar x) \le \langle \tilde{\mathcal{F}}_l, \bar{\mathcal{X}} \rangle$, $l=0,\ldots,p$. So, $\bar x$ is feasible for (P) and
$\min(P)=f_0(\bar x) \le \min(CP)$. Thus, the conclusion follows as $\min(P)=\min(CP).$
\end{proof}

It is worth noting that, in general, checking  the membership problem $\mathcal{X} \in U_{m,n+1}$ is, in general, an NP hard problem. Thus, solving the above conic programming problem is, in general, again a hard problem. This motivates us to examine an alternative tractable approach for solving nonconvex polynomial optimization problem with essentially non-positive  coefficients.

Below we show that the optimal value of the  nonconvex polynomial optimization problem with essentially non-positive  coefficients (P) can be
computed by the following sum-of-squares program:
\begin{eqnarray*}
(SOS)\ \ \  \max\{\mu &:& f_0+\sum_{l=1}^p \lambda_l f_l-\mu =\sigma_0,  \ \lambda_l \ge 0, \ l=1,\ldots,p, \ \sigma_0 \mbox{ is SOS},\ {\rm deg}\sigma_0 \le m\}. \ \ \ \ \ \ \ \ \ \ \ \ \ \ \ \ \ \ \ \ \ \ \ \ \ \ \ \ \  \ \ \ \ \ \ \ \ \ \ \ \ \ \ \ \ \ \ \ \ \ \ \ \ \ \ \ \ \ \ \ \ \ \ \ \ \ \ \ \ \ \ \ \ \ \ \ \ \ \ \ \ .
\end{eqnarray*}
We note that this problem can be regarded as the first level problem in the celebrated Lasserre hierarchy approximation of the general polynomial
 optimization problem. Moreover, the above sum-of-squares program can be equivalently reformulated as
a semidefinite programming problem. For details see the excellent surveys [38,39].

 \begin{theorem}{\bf (Exact Sums-of-Squares Relaxation)}\label{th:2}
  Let $n,p \in \mathbb{N}$ and let $m$ be an even number. Let $f_l$, $l=0,1,\ldots,p$, be polynomials on $\mathbb{R}^n$ with essentially nonpositive coefficients and degree $m$. Consider the  nonconvex polynomial optimization problem with essentially nonpositive  coefficients (P).
  Suppose that the
  strict feasibility condition holds, i.e., there exists $x_0 \in \mathbb{R}^n$ such that $f_l(x_0)<0$ for all $l=1,\ldots,p$.
 Then, we have
\begin{eqnarray*}
\min(P) = \max\{\mu &:& f_0+\sum_{l=1}^p \lambda_l f_l-\mu =\sigma_0, \\
 & & \lambda_l \ge 0, \ l=1,\ldots,p, \\
 & & \sigma_0 \mbox{ is SOS},\, {\rm deg}\sigma_0 \le m\},
\end{eqnarray*}
and the maximum in the sum-of-squares problem is attained.
 \end{theorem}
 \begin{proof}
We first observe that  \begin{eqnarray*}
\min(P) \ge \max\{\mu &:& f_0+\sum_{l=1}^p \lambda_l f_l-\mu =\sigma_0, \\
 & & \lambda_l \ge 0, \ l=1,\ldots,p, \\
 & & \sigma_0 \mbox{ is SOS},\, {\rm deg}\sigma_0 \le m\},
\end{eqnarray*}
always holds. To see the reverse inequality and the attainment, we can assume that $\min(P)>-\infty$. As the feasible set of (P) is nonempty, $\gamma:=\min(P) \in \mathbb{R}$. This implies that the following strict inequality system has no solution::
\[
x \in \mathbb{R}^n,\ f_l(x) < 0, \ l=1,\ldots,p \mbox{ and } f_0(x)-\gamma < 0.
\]
Let $\tilde{f}_l$ be the canonical homogenization of $f_l$, $l=0,1,\ldots,p$. From the definition of canonical homogenization,
$\tilde{f}_l(x,1)=f_l(x)$ for all $x \in \mathbb{R}^n$, $l=0,1,\ldots,p$.
We now see that the following strict homogeneous inequality system also has no solution:
\begin{equation}\label{eq:claim_homo}
(x^T,t)^T \in \mathbb{R}^{n+1}, \ \tilde{f}_l(x,t)< 0, \ l=1,\ldots,p \mbox{ and } \tilde{f}_0(x,t)-\gamma t^m < 0.
\end{equation}
Suppose on the contrary that there exists
$(\bar x^T,\bar t \, )^T \in \mathbb{R}^{n+1}$ such that $$\tilde{f}_l(\bar x,\bar t\, ) < 0, \ l=1,\ldots,p \mbox{ and }\tilde{f}_0(\bar x,\bar t)-\gamma \bar t^m < 0.$$
If $\bar t \neq 0$, then we have
\[ f_l(\frac{\bar x}{\bar t})=\tilde{f}_l(\frac{\bar x}{\bar t},1)=\frac{\tilde{f}_l(\bar x,\bar t)}{\bar t^m}< 0, \, l=1,\ldots,p \, \mbox{ and } \,  f_0(\frac{\bar x}{\bar t})-\gamma = \tilde{f}_0(\frac{\bar x}{\bar t},1)-\gamma=\frac{\tilde{f}_0(\bar x,\bar t)-\gamma \bar t^m}{\bar t^m}  < 0.\]
This makes contradiction. Now, if $\bar t=0$, then we have
\[
\tilde{f}_l(\bar x,0) < 0, \ l=1,\ldots,p \mbox{ and }\tilde{f}_0(\bar x,0) < 0.
\]
This implies that  $\lim_{\mu \rightarrow +\infty}f_l(\mu \bar x)=-\infty$, $l=0,1,\ldots,p,$ and so, for all large $\mu>0$,
\[
{f}_l(\mu \bar x) < 0, \ l=1,\ldots,p \mbox{ and } {f}_0(\mu \bar x) < \gamma.
\]
This also makes contradiction, and hence the claim (\ref{eq:claim_homo}) follows.

Letting $\mathcal{E}_{n+1}$ be the $m$th-order $(n+1)$-dimensional symmetric tensor such that $\langle \mathcal{E}_{n+1}, \tilde{x}^{\otimes m}\rangle =t^m$ for $\tilde{x}=(x^T,t)^T \in \mathbb{R}^{n+1}$ and noting that $\tilde{f}_l(\tilde{x})=\langle \tilde{\mathcal{F}}_l, \tilde{x}^{\otimes m}\rangle$ for any $\tilde{x}=(x^T,t)^T \in \mathbb{R}^{n+1}$, (\ref{eq:claim_homo}) gives us that the following system has no solution
\[
\langle \tilde{\mathcal{F}}_l, \tilde{x}^{\otimes m} \rangle < 0, \, l=1,\ldots,p \ \mbox{and} \ \langle \tilde{\mathcal{F}}_0-\gamma \mathcal{E}_{n+1}, \tilde{x}^{\otimes m} \rangle < 0.
\]
Note that $\mathcal{F}_l$, $l=0,1,\ldots,p$, are all essentially nonpositive tensors (and so, $\tilde{\mathcal{F}}_0-\gamma \mathcal{E}_{n+1}$ is also essentially nonpositive). Then, Theorem \ref{th:1} implies that there exist $\lambda_l \ge 0$, $l=0,1,\ldots,p$, such that $\sum_{l=0}^p\lambda_l=1$ and
\[
\lambda_0(\tilde{\mathcal{F}}_0-\gamma \mathcal{E}_{n+1})+\sum_{l=1}^p \lambda_l\tilde{\mathcal{F}}_l \in {\rm SOS}_{m,n+1}.
\]
This shows that
\begin{equation}\label{eq:ppq}
\tilde{\sigma}_0(x,t):=\langle \lambda_0(\tilde{\mathcal{F}}_0-\gamma \mathcal{E}_{n+1})+\sum_{l=1}^p \lambda_l\tilde{\mathcal{F}}_l, \tilde{x}^{\otimes m}\rangle=\lambda_0(\tilde{f}_0(x,t)-\gamma t^m) +\sum_{l=1}^p \lambda_l\tilde{f}_l(x,t)
\end{equation}
is a sum-of-squares polynomial with degree $m$. Letting $t=1$ in (\ref{eq:ppq}), it follows that
\begin{equation}\label{eq:final}
\lambda_0(f_0(x)-\gamma) +\sum_{l=1}^p \lambda_lf_l(x)=\tilde{\sigma}_0(x,1)
\end{equation}
is a sum-of-squares polynomial with degree $m$. We now show that $\lambda_0>0$. Indeed, if $\lambda_0=0$, then $\sum_{l=1}^p\lambda_l=1$ and
\[
\sum_{l=1}^p \lambda_lf_l(x)=\tilde{\sigma}_0(x,1) \ge 0\mbox{ for all } x \in \mathbb{R}^n.
\]
Thus, the strict feasibility condition implies that $\lambda_l=0$, $l=1,\ldots,p$. This contradicts the fact that $\sum_{l=1}^p\lambda_l=1$, and so, $\lambda_0>0$.  Dividing $\lambda_0$ on both sides of (\ref{eq:final}) shows that
\[
f_0(x)-\gamma +\sum_{l=1}^p \frac{\lambda_l}{\lambda_0}f_l(x)=\frac{\tilde{\sigma}_0(x,1)}{\lambda_0},
\]
is a sum-of-squares polynomial with degree $m$,
and so,
\begin{eqnarray*}
\min(P)=\gamma \le \max\{\mu &:& f_0+\sum_{l=1}^p \lambda_l f_l-\mu =\sigma_0, \\
 & & \lambda_l \ge 0, \ l=1,\ldots,p, \\
 & & \sigma_0 \mbox{ is SOS},\, {\rm deg}\sigma_0 \le m\}.
\end{eqnarray*}
Thus, the conclusion follows.
 \end{proof}
\begin{remark} {\bf (Connection to the existing result in polynomial optimization)}
{\rm It is known that the optimal value of a general nonconvex polynomial optimization problem can be approximated by a sequence of semidefinite programming problem under the so-called Archimedean assumption. We note that the Archimedean assumption implies the feasible set of the nonconvex polynomial optimization problem must be compact. This sequence of semidefinite programming problem is now often referred as Lasserre hierarchy and becomes one of the important and popular tools in solving a general polynomial optimization problem with compact feasible sets. For excellent
survey see [37,38,46]. 
It is worth noting that, if we have some prior knowledge about a solution $x^*$ (say $\|x^*\| \le R$ for some $R>0$), one can
impose an additional constraint $\|x\|^2 \le R^2$ and convert the problem into an optimization problem with compact feasible set. In this case, a global solution can be found by
using this big ball approach as long as we have some prior knowledge about a solution. Moreover, there are also some other approaches for solving polynomial optimization problems with unbounded feasible sets
by exploiting gradient ideals of the underlying problem (for example, see
[48-50]). 

On the other hand, Theorem \ref{th:2} shows that the optimal value of a nonconvex polynomial optimization problem with essentially nonpositive coefficients can be found by solving the first level problem in the Lasserre hierarchy approximation under the strict feasibility condition. Interestingly, Theorem \ref{th:2} allows the feasible set to be non-compact (see
Example \ref{ex:2}) without having prior knowledge of the solution $x^*$.}
\end{remark}
As a corollary, we show that the nonconvex polynomial optimization problem with generalized $l^m$-type constraints enjoys exact sum-of-squares relaxation whenever the objective function has essentially nonpositive coefficients.

\begin{corollary}{\bf (Exact Sums-of-Squares Relaxation for generalized $l^m$ constraints)} \label{cor:1}
Let $m$ be an even number and let $n \in \mathbb{N}$. Let $f_0$ be a polynomial on $\mathbb{R}^n$ with essentially nonpositive coefficients and degree $m$. Consider the following nonconvex polynomial optimization problem with generalized $l^m$-type constraint:
\begin{eqnarray*}
(P_{l^m}) \ \ \ \ \ \ \ \ \ \ \ \ \ \ \ \ \ \ \ \ \ \ \ \ \ \ \ \ & \displaystyle \min_{x \in \mathbb{R}^n} \{f_0(x):  \sum_{i=1}^na_i x_i^m \le 1\},  \ \ \ \ \ \ \ \ \ \ \ \ \ \ \ \ \ \ \ \ \ \ \ \ \ \ \ \ \  \ \ \ \ \ \ \ \ \ \ \ \ \ \ \ \ \ \ \ \ \ \ \ \ \ \ \ \ \ \ \ \ \ \ \ \ \ \ \ \ \ \ \ \ \ \ \ \ \ \ \ \
\end{eqnarray*}
where $a_i \in \mathbb{R}$, $i=1,\ldots,n$. Let $f_1(x)=\sum_{i=1}^na_i x_i^m-1$.
 Then, we have
\begin{eqnarray*}
\min(P_{l^m}) = \max\{\mu &:& f_0+\lambda f_1 -\mu =\sigma_0, \\
 & & \lambda \ge 0, \mu \in \mathbb{R} \\
 & & \sigma_0 \mbox{ is SOS},\, {\rm deg}\sigma_0 \le m\},
\end{eqnarray*}
and the maximum in the sum-of-squares problem is attained.
\end{corollary}
\begin{proof}
Clearly, $f_1(0)=-1$ and so, the strict feasibility condition is satisfied for $(P_{l^m})$. So, the conclusion follows from the preceding theorem.
\end{proof}

\begin{remark}{\bf (Further links to the existing literature)}  {\rm Below, we compare the preceding corollary with some known results in the literature.
\begin{itemize}
\item[{\rm (1)}] We first discuss the relationship of problem $(P_{l^m})$ and the positive-definiteness problem of a symmetric tensor. Let $\mathcal{A}$ be a symmetric tensor and let $f_{\mathcal{A}}(x)=\langle \mathcal{A},x^{\otimes m}\rangle$. We say the tensor $\mathcal{A}$ is positive definite if $f_{\mathcal{A}}(x)>0$ for all $x \in \mathbb{R}^n \backslash \{0\}$. This is equivalent to the fact that the optimal value of the following polynomial optimization problem is positive:
\begin{eqnarray*}
\displaystyle \min_{x \in \mathbb{R}^n} \{ f_{\mathcal{A}}(x) :\sum_{i=1}^n x_i^m \le 1\}.
\end{eqnarray*}
Note that this is a special case of $P_{l^m}$ with homogeneous objective function and $a_i=1$. So, the preceding corollary shows that
the positive definiteness of an essentially non-positive tensor can be tested by solving a sum-of-squares programming problem.  This result has been established very recently
 in [31].
\item[{\rm (2)}] Recently, the nonconvex polynomial optimization problem with generalized $l^m$-type constraint was studied in [51] and a geometric programming relaxation problem was proposed to calculate the lower bound of the optimal value of the problem (see also [52]). It was demonstrated that the lower bound provided by the geometric programming relaxation is a lower bound of the sum-of-squares relaxation and can be more efficient from the computational point of view  comparing
to the sum-of-squares relaxation. At this moment, it is not clear for us whether the geometric programming relaxation is indeed exact in the case when the objective function is a polynomial with essentially nonpositive coefficients. This would be an interesting research question for our further study.
\end{itemize}}

\end{remark}

Before we end this section, we provide an example verifying Theorem  \ref{th:2} and Corollary \ref{cor:1}.
\begin{example}\label{ex:3}
 Let $f_0$ be a  homogeneous polynomial on $\mathbb{R}^3$ with degree $6$ defined by $$f_0(x_1,x_2,x_3)=x_1^6+x_2^6+x_3^6-(x_1^2(x_2^4+x_3^4)+x_2^2(x_1^4+x_3^4)+x_3^2(x_1^4+x_2^4)).$$
 Clearly $f$ is a homogeneous polynomial with essentially non-positive coefficients. Consider the homogeneous polynomial optimization problem \begin{eqnarray*}
(EP3) \ \ \ \ \ \ \ \ \ \ \ \ \ \  \ \ \ \ \  \ \ \ \ & \displaystyle \min_{x \in \mathbb{R}^3}\{f_0(x): x_1^6+x_2^6+x_3^6 \le 1\} \ \ \ \ \ \ \ \ \ \ \ \ \ \ \ \ \ \ \ \ \ \ \ \ \ \ \ \ \  \ \ \ \ \ \ \ \ \ \ \ \ \ \ \ \ \ \ \ \ \ \ \ \ \ \ \ \ \ \ \ \ \ \ \ \ \ \ \ \ \ \ \ \ \ \ \ \ \ \ \ \ .
\end{eqnarray*}
Let $f_1(x)=x_1^6+x_2^6+x_3^6-1$. Then, the corresponding sum-of-squares relaxation is given by
\begin{eqnarray*}
(REP3) & \ \ \ \ \ \ \ \ \ \ \ \ \ \ \ \ \ \ & \max_{\lambda \ge 0, \mu \in \mathbb{R}}\{\mu: f_0+\lambda f_1-\mu=\sigma_0, \, \sigma_0 \mbox{ is SOS and } {\rm deg} \sigma_0 \le 6 \} \ \ \ \ \ \ \ \ \ \ \ \ \ \ \ \ \ \ \ \ \ \ \ \ \ \ \ \ \  \ \ \ \ \ \ \ \ \ \ \ \ \ \ \ \ \ \ \ \ \ \ \ \ \ \ \ \ \ \ \ \ \ \ \ \ \ \ \ \ \ \ \ \ \ \ \ \ \ \ \ \ 
\end{eqnarray*}
Solving the sum of squares programming problem (REP3) via YALMIP {\rm (see [53,54])} gives us that $\min(REP3)=-1$.

On the other hand, note that the following Robinson polynomial {\rm (cf. [41])}
\[
f_R(x)=x_1^6+x_2^6+x_3^6-(x_1^2(x_2^4+x_3^4)+x_2^2(x_1^4+x_3^4)+x_3^2(x_1^4+x_2^4))+3x_1^2x_2^2x_3^2
\]
is always non-negative, and $f_R(x)=f_0(x)+3x_1^2x_2^2x_3^2$. This implies that, for all $(x_1,x_2,x_3)$ with $x_1^6+x_2^6+x_3^6 \le 1$,
\[
f_0(x)  \ge -3x_1^2x_2^2x_3^2 \ge -3 (\frac{(x_1^2)^3+(x_2^2)^3+(x_3^2)^3}{3}) \ge -1,
\]
where the second inequality follows by the inequality of arithmetic and geometric means.
Moreover, note that $(\bar x_1,\bar x_2,\bar x_3)=(\sqrt[6]{\frac{1}{3}},\sqrt[6]{\frac{1}{3}},\sqrt[6]{\frac{1}{3}})$ satisfies $\bar x_1^6+\bar x_2^6+\bar x_3^6 \le 1$ and
$f_0(\bar x)=-1$. So, $\min(EP3)=-1$. This verifies the sum-of-squares relaxation is exact for this example.
\end{example}
\subsection{Examples}
Below, we present a few numerical examples. The first example shows that the conclusion of Theorem \ref{th:2} can fail if a polynomial optimization problem does not have  essentially nonpositive coefficients. The second example illustrate that Theorem \ref{th:2} can be applied to a polynomial optimization problem with possibly non-compact feasible set.
\begin{example}\label{ex:1} {\bf (Importance of the assumption on essentially nonpositive coefficients)}
Let $f_M$ be the homogeneous Motzkin polynomial
\[f_M(x)=x_3^6 + x_1^2x_2^4 + x_1^4x_2^2-3x_1^2x_2^2x_3^2.
\]
It is clear that $f_M$ is not a polynomial with essentially nonpositive coefficients. It is known that $f_M$ is a  polynomial which takes non-negative values but is not a sum-of-squares polynomial {\rm [41]}.
Consider the following polynomial optimization problem
\begin{eqnarray*}
(EP1)  \ \ \ \ \  \ \ \ \ \  \ \ \ \ \  & \displaystyle
\min_{x \in \mathbb{R}^3} \{f_M(x): x_1^6+x_2^6+x_3^6 \le 1\}  \ \ \ \ \   \ \ \ \ \   \ \ \ \ \   \ \ \ \ \   \ \ \ \ \   \ \ \ \ \  \ \ \ \ \   \ \ \ \ \   \ \ \ \ \  \ \ \ \ \  \ \ \ \ \  \ \ \ \ \  .
\end{eqnarray*}
Clearly,  $\min(EP1)=0$ (as $f_M$ takes non-negative value).
Let $f_1(x)=x_1^6+x_2^6+x_3^6-1$.
The sum-of-squares relaxation of (EP1) takes the form
\begin{eqnarray*}
 (REP1) & \ \ \ \ \   \ \ \ \ \     & \max_{\lambda \ge 0, \mu \in \mathbb{R}}\{\mu: f_M+\lambda f_1-\mu=\sigma_0, \, \sigma_0 \mbox{ is SOS and } {\rm deg} \sigma_0 \le 6\}  \ \ \ \ \   \ \ \ \ \   \ \ \ \ \   \ \ \ \ \   \ \ \ \ \   \ \ \ \ \  \ \ \ \ \   \ \ \ \ \   \ \ \ \ \  \ \ \ \ \  \ \ \ \ \  \ \ \ \ \   \ \ \ \ \   \ \ \ \ \   \ \ \ \ \   \ \ \ \ \   \ \ \ \ \   \ \ \ \ \  \ \ \ \ \   \ \ \ \ \   \ \ \ \ \  \ \ \ \ \  \ \ \ \ \  \ \ \ \ \ 
\end{eqnarray*}
We now show that the conclusion of Theorem \ref{th:2} fails. To see this, we suppose on the contrary that there exist $\lambda \ge 0$ and a  sum-of-squares polynomial $\sigma_0$ with degree at most $6$ such that
$f_M+\lambda f_1=\sigma_0$.   Note that $0 \le \sigma_0(0)=f_M(0)+\lambda f_1(0)=-\lambda$. This together with $\lambda \ge 0$ implies that $\lambda=0$, and so, $f_M=\sigma_0$ which is a sum-of-squares polynomial. This contradicts the fact that $f_M$ is not a sum-of-squares polynomial. Thus the conclusion of Theorem \ref{th:2} fails.
%
\end{example}

\begin{example}\label{ex:2}{\bf (An example with noncompact feasible set)}
Let $f_0$ be a  homogeneous polynomial on $\mathbb{R}^3$ with degree $4$ defined by $f_0(x_1,x_2,x_3)=x_1^4+x_2^4+x_3^4-4x_1x_3^3$. Clearly, $f$ is a homogeneous polynomial with essentially non-positive coefficients. Consider the homogeneous polynomial optimization problem \begin{eqnarray*}
(EP2)  \ \ \ \ \   \ \ \ \ \  & \displaystyle \min_{x \in \mathbb{R}^3} \{f_0(x): x_1^4-\frac{1}{2}x_2^4+x_3^4 \le 1\}.   \ \ \ \ \   \ \ \ \ \   \ \ \ \ \   \ \ \ \ \   \ \ \ \ \   \ \ \ \ \  \ \ \ \ \   \ \ \ \ \   \ \ \ \ \  \ \ \ \ \  \ \ \ \ \  \ \ \ \ \
\end{eqnarray*}
Clearly, the feasible set of (EP2) is not compact.  Let $f_1(x)=x_1^4-\frac{1}{2}x_2^4+x_3^4-1$. Then, the corresponding sum-of-squares relaxation is given by
\begin{eqnarray*}
 (REP2) \ \ \ \ \   \ \ \ \   & \displaystyle \max_{\lambda \ge 0, \mu \in \mathbb{R}}\{\mu: f_0+\lambda f_1-\mu=\sigma_0, \, \sigma_0 \mbox{ is SOS and } {\rm deg} \sigma_0 \le 4 \}   \ \ \ \ \   \ \ \ \ \   \ \ \ \ \   \ \ \ \ \   \ \ \ \ \   \ \ \ \ \  \ \ \ \ \   \ \ \ \ \   \ \ \ \ \  \ \ \ \ \  \ \ \ \ \  \ \ \ \ \ 
\end{eqnarray*}
Solving the sum of squares programming problem (REP2) via YALMIP $($see {\rm [52,53]}$)$ gives us that $\min(REP2)=-1.2795$. 

On the other hand, direct calculation shows that for any global minimizer of (EP2) satisfies the following KKT condition: there exist $\lambda \ge 0$ and
 $(x_1,x_2,x_3)$ with $x_1^4-\frac{1}{2}x_2^4+x_3^4 \le 1$ such that
\begin{eqnarray*}
\left\{\begin{array}{lcc}
x_1^3-x_3^3+\lambda x_1^3 &=& 0, \\
x_2^3-\frac{\lambda}{2} x_2^3 &= & 0, \\
x_3^3-3x_1x_3^2 + \lambda x_3^3 & = & 0.
\end{array} \right.
\end{eqnarray*}
Solving this homogeneous polynomial equality system gives us that $\lambda=2$ or $\lambda=\sqrt[4]{27}-1$ and the possible KKT points are  $$\{(0,x_2,0): x_2 \in \mathbb{R}\} \cup \{(x_1,0,x_3): x_3=\sqrt[4]{3} \ x_1, \, |x_1| \le \sqrt[4]{\frac{1}{4}}\}.$$
 By comparing the corresponding objective function values of the KKT points, it can be verified that the optimal value of (EP2) is $1-\sqrt[4]{27} \approx -1.2795.$ This verifies that the sum-of-squares relaxation is exact.
\end{example}

\section{Perspectives}
Alternative theorems for arbitrary finite systems of linear or convex inequalities
have played key roles in the development of optimality conditions for continuous optimization
problems. Although these theorems are generally not valid for an
arbitrary finite system of (possibly nonconvex) quadratic inequalities, recent research has established alternative
theorems for quadratic systems involving two inequalities or arbitrary inequalities involving suitable sign structure. For
instance, a theorem of the alternative of Gordan type for a strict inequality system of two
homogeneous quadratic functions has been given in [1], where it was used in convergence
analysis of trust-region algorithms. This Theorem is often referred as Yuan's theorem of the alternative and has closed connection with
 the convexity of joint-range of homogeneous quadratic functions even though the functions may be non-convex [3,4].

On the other hand, tensor computation and optimization problems involving polynomials arise in a wide variety of contexts, including operational research, statistics, probability, finance, computer science, structural engineering, statistical physics,  computational biology and  graph theory [33,37,38]. They are however extremely challenging to solve, both in theory and practice.
A fascinating feature of this field is that it can be approached from several different directions. In addition to traditional techniques drawn from operational research, computer science and numerical analysis, new techniques have recently emerged based on concepts taken from algebraic geometry, moment theory, multilinear algebra and modern convex programming (semidefinite programming).

Due to the wide application of  tensor computation and polynomial optimization, an important research topic is to obtain a tractable extension of alternative
theorem for system of homogeneous polynomial inequalities (or equivalently inequalities involving tensors). Obtaining such a multilinear (or tensor) version of an theorem of the alternative is extremely useful as it naturally
leads to numerically checkable conditions for a global minimizer of a related polynomial optimization problem. Unfortunately, in general, this is an extreme challenging task. Two of the main
obstructions are (1) some of the nice geometric structure (such as joint-range convexity) for homogeneous quadratic functions cannot be carried forward to polynomial cases, and are more challenging to
exploit (2) unlike the quadratic cases, checking the nonnegativity of a homogeneous polynomial (or equivalently the positive semi-definiteness of a symmetric tensor) is, in general, an NP-hard problem [13,37,38].

In this paper, we provided a tractable extension of Yuan's theorem of the alternative in the symmetric tensor setting.
We achieve this by exploiting two important features of a special class of tensors (called essentially non-positive tensors): hidden convexity and numerical checkability.
As
an immediate application, we showed that the optimal value and optimal solution of a nonconvex polynomial optimization problem with essentially nonpositive coefficients can be found by a related convex conic programming problem. We also established that this class of polynomial optimization problem enjoys exact sum-of-squares relaxation.

Our results point out some useful observations and interesting further research topics. In particular, although a tensor problem (or a polynomial optimization problem)
is, in general NP-hard, we feel that it is important to exploit the special structure of the underlying problem and  push the boundary of the tractable classes of
problems. This is of particular importance because (1) those tractable classes are the problems we can efficiently solve via the current software/technology; (2) many of the practical problems often come with some special structures (such as sign structure and sparse structure)
naturally. The results presented in this paper suggest that problems involving tensors/polynomials with suitable sign structure would be a good candidate for the tractable
classes. In fact, this is not a coincidence as it was shown recently that almost the whole Perron-Frobenius theory for non-negative matrices can be extended to
tensor setting, and so, the extreme eigenvalue problem involving tensors with non-negative entries is numerically tractable [25,29,35,36]. On the other hand,
 this paper is still a preliminary study for structured tensors (or polynomial optimization with special structures) and a lot of interesting research topics need further investigation. Below, we list some
 of the topics which are particularly important from our point of view:

 \begin{itemize}
  \item[(a)]  Can one extend the results presented in this paper to a special structured tensor other than the tensors with essentially non-positive entries? Some particularly
  important structured tensors arise naturally in signal processing, stochastic process and data fitting include the Hankel tensors and circulant tensors [54-56]. Can theorem of the alternatives be extended
  to cover these structured tensors?
  \item[(b)]  As discussed in Example \ref{ex:1}, our exact relaxation result can fail for a nonconvex polynomial optimization problems if the functions involved do not have essentially nonpositive coefficients. On the other hand, it would be of interest to see how our results can be used to provide some approximate bounds for the optimal value of the general nonconvex polynomial optimization problems.

\item[(c)]  Finally, it would be also useful to extend the known theorem of the alternative for copositive matrix to the symmetric tensor setting (if possible).
 \end{itemize}
These will be our future research directions.
\medskip

\section{Conclusion}
In this paper, by exploiting the hidden convexity and numerical checkability of a special class of tensors, we established a tractable extension of Yuan's theorem of alternative in the symmetric tensor setting. As an application, we showed that the solution of a polynomial optimization
problem with suitable structure can be found by solving a single semi-definite programming problem.

\bigskip

\noindent{{\bf Acknowledgement:} The authors would like to express their sincere thanks the referees for their constructive comments and valuable suggestions, which have contributed to the revision of this paper. Moreover, the second author would like to thank Prof. J.B. Lasserre and Prof. T.S. Pham for pointing out the related references [50,51] during their visit in UNSW.

\bigskip 

\noindent Research was partially supported by the Australian Research Council Future Fellowship (FT130100038) and the Hong Kong Research Grant Council (Grant No. PolyU
502510, 502111, 501212 and 501913), and National Natural Science
Foundation of China (Grant No. 11101303).

\noindent \section*{References}

\begin{enumerate}[1.]

\item 
Yuan, Y.X.: On a subproblem of trust region algorithms for constrained optimization, {Math. Prog.}, 47, 53-63,  (1990).

\item 
Yan, Z.Z., Guo, J.H.:  Some equivalent results with Yakubovich's S-lemma. {SIAM J. Control Optim.} 48, no. 7, 4474-4480,  (2010).

  \item 
Jeyakumar,   V. , Huy  H.Q., Li, G.: Necessary and sufficient conditions for S-lemma and nonconvex quadratic optimization, {Optim.  Eng.},
10, 491-503,  (2009).

 \item 
P\'{o}lik,  I., Terlaky, T.,  A survey of the S-Lemma,   {SIAM Review}, {\bf 49}, 371-418,  (2007).

 \item 
Sturm  J. F., Zhang, S. Z.:  On cones of
 non-negative quadratic functions,     {Math. Oper. Res.}, 28, 246-267,  (2003).

\item 
Yakubovich, V. A.: S-Procedure in nonlinear control theory, {Vestnik Leningrad. Univ.}, 1,  62-77,  (1971).

\item 
Chen X., Yuan, Y.: A note on quadratic forms, {Math. Program.}, 86, 187-197, (1999).

\item 
Polyak, B.T.:  Convexity of quadratic transformation and its use in control and optimization,
{J. Optim. Theory Appl.}, {99}, 563-583,  (1998).

\item 
Crouzeix, J.P.,  Martinez-Legaz, J.E., Seeger, A.:  An theorem of the alternative for quadratic forms and extensions. {Linear Algebra Appl.} 215 , 121-134, (1995).

\item 
Mart\'{i}nez-Legaz, J. E., Seeger, A.: Yuan's theorem of the alternative and the maximization of the minimum eigenvalue function.  {J. Optim. Theory Appl.}  {82},  no. 1, 159-167,  (1994).

\item 
Jeyakumar, V., Lee, G.M., Li, G.:  Alternative theorems for quadratic inequality systems and global quadratic optimization, {SIAM J.  Optim.}, 20, no. 2, 983-1001,  (2009).

 \item 
Lim,  L.H.:  Singular values and eigenvalues of tensors, A variational approach,
{Proc. 1st IEEE International workshop on computational advances of multi-tensor adaptive processing}, 129-132,  (2005).

\item 
Qi, L.: Eigenvalues of a real symmetric tensor,
{J. Symb. Comp.,} 40, 1302-1324, (2005).

\item   Bomze, I.M., Ling, C., Qi, L., Zhang, X.: Standard bi-quadratic
optimization problems and unconstrained polynomial reformulations,
{J. Glob. Optim.},  52, 663-687, (2012).

\item 
He, S., Li Z., Zhang, S.: Approximation algorithms for homogeneous polynomial optimization with quadratic
constraints, {Math. Prog.} 125, 325-383,  (2010).

\item
Zhang, X.,  Qi,  L., Ye, Y.:  The cubic spherical optimization
problems, {Math. Comp.} 81, 1513-1525,  (2012).

\item 
Ling, C., Nie,  J., Qi   L., Ye, Y.: Bi-quadratic optimization over
unit spheres and semidefinite programming relaxations, {SIAM
J. Optim.} 20, 1286-1310, (2009).

\item 
 So, A. M-C.: Deterministic approximation algorithms for sphere constrained homogeneous polynomial optimization
problems, {Math. Prog.} 129, 357-382,  (2011).

\item Li, G., Mordukhovich, B.S., Pham, T.S.: New fractional error bounds for polynomial systems with applications to H\"{o}lderian stability in optimization and spectral theory of tensors, to appear in {Math. Prog.},
DOI: 10.1007/s10107-014-0806-9.

\item 
Qi, L., Xu, Y., Yuan, Y.,  Zhang, X.: A cone constrained convex
program: structure and algorithms, {J. Oper. Res. Society China}, 1, 37-53, (2013).

\item 
Qi L., Ye, Y.: Space tensor conic programming, to appear in: {Comp. Optim. Appl.}, (2013).

\item Cooper, J., Dutle, A.: Spectral of hypergraphs, {Linear Algebra Appl.}, 436, 3268-3292,  (2012).

\item 
Hu, S., Qi, L.:  Algebraic connectivity of an even uniform
hypergraph, {J. Comb. Optim.}, 24, 564-579,  (2012).

\item 
Li, G., Qi L., Yu, G.: The Z-eigenvalues of a symmetric tensor and its application to spectral hypergraph theory, {Num. Linear Algebra  Appl.}, 20, no. 6, 1001-1029, (2013).

\item 
Qi, L.:  H$^+$-eigenvalues of Laplacian and signless
Laplacian tensors, {Comm.
Math. Sci.}, 12, 1045-1064,  (2014).






\item 
Ng, M., Qi L., Zhou, G.:  Finding the largest eigenvalue of a non-negative tensor,
{SIAM J. Matrix Anal. Appl.}, 31, 1090-1099, (2009).

\item 
Lathauwer, L. De, Moor, B.: {From matrix to tensor: Multilinear algebra and signal processing}. In: J. McWhirter, Editor, Mathematics in Signal Processing IV, Selected papers presented at 4th IMA Int. Conf. on Mathematics in Signal Processing, Oxford University Press, Oxford, United Kingdom, 1-15,  (1998).

\item 
Qi, L., Teo, K.L.:  Multivariate polynomial minimization and its application in signal processing, {J. Global Optim.}, 46, 419-433, (2003).

\item 
Qi, L.,  Yu  G., Wu, E.X.:  Higher order positive semi-definite diffusion tensor
imaging,  {SIAM J. Imaging Sci.}, 3, 416-433,  (2010).

\item 
Chang, K.C.,  Pearson, K., Zhang, T.: Primitivity, the convergence of the NZQ
method, and the largest eigenvalue for non-negative tensors, {SIAM J. Matrix Anal. Appl.}, 32, 806-819, (2011).

\item 
Hu, S.,  Li, G.,  Qi,  L., Song, Y.: Finding the maximum eigenvalue of essentially non-negative symmetric tensors via sum of squares programming, {J. Optim. Theory Appl.}, 158, no. 3, 713-738,  (2013).

\item 
Kofidis E., Regalia, Ph.: On the best rank-1 approximation of higher-order symmetric tensors, {SIAM J. Matrix Anal. Appl.} 23, 863-884, (2002).

\item 
Kolda T.G., Bader, B.W.:  Tensor decompositions and applications.  {SIAM Review},  51,  no. 3, 455-500,   (2009).

\item 
Li, G., Qi L., Yu, G.:  Semismoothness of the maximum eigenvalue function of a symmetric tensor and its application,  {Linear Algebra Appl.},  438, 813-833, (2013).

\item 
Liu, Y., Zhou, G., Ibrahim, N.F.:  An always convergent algorithm for the largest eigenvalue of an irreducible non-negative tensor. {J. Comp. Applied Math.} 235, no. 1, 286-292,  (2010).

\item 
Zhang, L., Qi, L., Luo,  Z., Xu, Y.:  The dominant eigenvalue of an essentially non-negative tensor, {Num. Linear Algebra Appl.} 20, no. 6, 929-941,  (2013).

\item 
Yang, Y.N., Yang,  Q.Z.:  Further results for Perron-Frobenius theorem for non-negative tensors. {SIAM J. Matrix Anal. Appl.} 31, no. 5, 2517-2530,  (2010).

\item 
Lasserre, J.B.: {Moments, Positive Polynomials and their Applications}, Imperial College Press, (2009).

\item 
Laurent, M.:   {Sums of squares, moment matrices and optimization over polynomials}.  Emerging Applications of Algebraic Geometry, Vol. 149 of IMA Volumes in Mathematics and its Applications, M. Putinar and S. Sullivant (eds.), Springer, pages 157-270, (2009).

\item 
 Parrilo, P.A.: Semidefinite programming relaxations for semialgebraic problems.
 {Math. Prog.} Ser. B,  96, no.2, 293-320, (2003).

\item 
Hilbert, D.: \"{U}ber die Darstellung definiter Formen als Summe von Formenquadraten, {Math. Annalen}, 32, 342-350,  (1888).

 \item 
Reznick, B.:   Some concrete aspects of Hilbert's 17th Problem.  {Real algebraic geometry and ordered structures} (Baton Rouge, LA, 1996),   Contemp. Math., 253, Amer. Math. Soc., Providence, RI, 251-272, (2000).

\item 
Reznick, B.:   Sums of Even Powers of Real Linear Forms,
{Memoirs American Math. Society}, 96, no.
463, (1992).





\item 
Fidalgo, C., Kovacec, A.: Positive semidefinite diagonal minus tail forms are sums of squares,
{Mathe. Zeit.}
269, 629-645,  (2011).

\item 
Zalinescu, C.:  {Convex Analysis in General Vector
Spaces}, World Scientific, (2002).

\item 

Friedgut, E.: Hypergraphs, entropy, and inequalities.
{Amer. Math. Monthly} 111, 749-760,  (2004).

\item 
Lasserre,  J.B.: Global optimization with polynomials and the problem of
moments, {SIAM J. Optim.},  { 11}, 796-817,  (2001).

\item 
H\'{a}, H.V., Pham, T.S.:  Representations of positive polynomials and optimization on noncompact semialgebraic sets. {SIAM J. Optim.} 20, no. 6, 3082-3103, (2010).

\item 
Nie, J.W., Demmel, J., Sturmfels, B.: Minimizing polynomials via sum of squares over the gradient ideal. {Math. Prog.},  Ser. A, 106, no. 3, 587-606,  (2006).





\item 
Schweighofer, M.:
Global optimization of polynomials using gradient tentacles and sums of squares,
{SIAM J. Optim.} 17, no. 3, 920-942, (2006).

\item 
Ghasemi, M., Lasserre,  J. B., Marshall,  M.: Lower bounds on the global minimum of a polynomial, arXiv:1209.3049.

\item 
Ghasemi, M., J. B., Marshall,  M.:Lower bounds for polynomials using geometric programming,
{SIAM J. Optim.} 22, 460-473, (2012).

\item 
 L\"{o}fberg, J.: Pre- and post-processing sum-of-squares programs in practice, {IEEE Tran. Auto. Control}, 54, 1007-1011,  (2009).

\item 
L\"{o}fberg, J.: YALMIP: A Toolbox for Modeling and Optimization in MATLAB.  In Proceedings of the CACSD Conference, Taipei, Taiwan, (2004).

\item
Papy, J.M.,De Lathauwer, L., Van Huffel, S.:  Exponential data fitting using multilinear algebra: The single-channel and multi-channel case,
{Num. Linear Algebra Appl.}, 12, 809-826, (2005).

\item Ding,  W., Qi L.,  Wei, Y.: Fast Hankel tensor-vector products and application to exponential data fitting, January 2014. arXiv: 1401.6238.

\item Chen Z., Qi, L.: Circulant tensors with applications to spectral hypergraph theory and stochastic process, April 2014.   arXiv:1312.2752.
\end{enumerate}

\vspace{-0.6cm}
\section*{Appendix}

Proof of Proposition \ref{prop:0.1}

\begin{proof}
As any sum-of-squares polynomial takes non-negative value, ${\rm SOS}_{m,n} \cap E_{m,n} \subseteq {\rm PSD}_{m,n} \cap E_{m,n}$ always holds. We only need to show the converse inclusion. To establish this, let $\mathcal{A} \in {\rm PSD}_{m,n} \cap E_{m,n}$ and consider
the associated homogeneous polynomial
\[
f(x)=\langle \mathcal{A}, x^{\otimes m}\rangle=\sum_{i_1,\ldots,i_m=1}^{n} \mathcal{A}_{i_1\cdots
i_m}x_{i_1}\cdots x_{i_m}.
\]
Then, $f$ is a polynomial which takes non-negative value.
Note that
\[
f(x)=\sum_{i_1,\ldots,i_m=1}^{n} \mathcal{A}_{i_1\cdots
i_m}x_{i_1}\cdots x_{i_m}=\sum_{i=1}^n(\mathcal{A}_{ii\cdots i})x_i^{m}+\sum_{(i_1,\ldots,i_m) \notin I}(\mathcal{A}_{i_1\cdots
i_m})x_{i_1} \cdots x_{i_m},
\]
where $I:=\{(i,i,\ldots,i) \in \mathbb{N}^m: 1 \le i \le n\}.$
As $\mathcal{A}$ is essentially nonpositive, $\mathcal{A}_{i_1i_2\cdots
i_m} \le 0$ for all $(i_1,\ldots,i_m) \notin I$. Now, let $f(x)=\sum_{i=1}^n f_{m,i} x_i^{m}+\sum_{\alpha \in \Omega_f}f_{\alpha}x^{\alpha}$. Then, $f_{m,i}=\mathcal{A}_{ii\cdots i}$
 and $f_{\alpha} < 0$ for all $\alpha \in \Omega_f$ where $\Omega_f=\{\alpha=(\alpha_1,\ldots,\alpha_n) \in (\mathbb{N}\cup \{0\})^n: f_{\alpha} \neq 0 \mbox{ and } \alpha \neq m e_i, \ i=1,\ldots,n\},$ and $e_i$ is the vector where its $i$th component is one and all the other components are zero. Recall that
$\Delta_f  =  \{\alpha=(\alpha_1,\ldots,\alpha_n) \in \Omega_f: f_{\alpha} < 0 \mbox{ or } \alpha \notin (2\mathbb{N}\cup \{0\})^n\}.$
 Note that $f_{\alpha}<0$ for all $\alpha \in \Omega_f$ and so, $\Delta_f=\Omega_f$.  It follows that
\begin{eqnarray*}
\hat{f}(x)&:= &\sum_{i=1}^n f_{m,i} x_i^{m}-\sum_{\alpha \in \Delta_f}|f_{\alpha}|x^{\alpha}\\
&=& \sum_{i=1}^n f_{m,i} x_i^{m}+\sum_{\alpha \in \Delta_f}f_{\alpha}x^{\alpha} \\
&= &  \sum_{i=1}^n f_{m,i} x_i^{m}+\sum_{\alpha \in \Omega_f}f_{\alpha}x^{\alpha} = f(x).
\end{eqnarray*}
So, $\hat{f}$ is also a polynomial which takes non-negative value. Thus the conclusion follows by Lemma \ref{lemma:2.1}.
\end{proof}
\end{document}